\documentclass[11pt, reqno]{amsart}
\usepackage{amssymb,latexsym,amsmath,amsfonts,mathdots,enumitem}
\usepackage{latexsym}
\usepackage[mathscr]{eucal}
\usepackage{colortbl,xcolor}
\usepackage{lmodern}
\usepackage{sansmathaccent}
\usepackage[latin1]{inputenc}
\usepackage{tikz}
\usepackage{physics}
\usetikzlibrary{shapes,arrows}
\usetikzlibrary{matrix,calc,shapes,arrows,positioning}
\pdfmapfile{+sansmathaccent.map}

\numberwithin{equation}{section}
\theoremstyle{plain}

\newtheorem{thm}{Theorem}[section]

\newtheorem{lem}[thm]{Lemma}
\newtheorem{prop}[thm]{Proposition}
\theoremstyle{definition}
\newtheorem{defn}[thm]{Definition}
\newtheorem{rem}[thm]{Remark}
\newtheorem{ex}[thm]{Example}

\numberwithin{equation}{section}
\def\D{{\mathcal D}}

\def\I{\mathcal{I}}

\def\Re{\operatorname{Re}}

\def\beq{\begin{eqnarray}}
	\def\eeq{\end{eqnarray}}
\def\beqa{\begin{eqnarray*}}
	\def\eeqa{\end{eqnarray*}}

\def\Span{\operatorname{Span}}
\def\Ran{\operatorname{Ran}}

\def\beqn{\begin{equation}}
	\def\eeqn{\end{equation}}

\def\mg#1{}

\def\Ran{\operatorname{Ran}}

\renewcommand{\epsilon}{\varepsilon}
\renewcommand{\phi}{\varphi}

\renewcommand{\bf}[1]{\textbf{#1}}
\renewcommand{\it}[1]{\textit{#1}}

\renewcommand{\sf}[1]{\textsf{#1}}

\renewcommand{\Re}[1]{\sf{Re}(#1)}

\numberwithin{equation}{section}%Code for numbering equations sectionwise
\allowdisplaybreaks[4] %If you prefer a strategy of letting page breaks fall where they may, even in the middle of a multi-line equation, then you might put \allowdisplaybreaks[1] in the preamble of your document. An optional argument 1ÃÂÃÂ¢ÃÂÃÂÃÂÃÂ4 can be used for finer control: [1] means allow page breaks, but avoid them as much as possible; values of 2,3,4 mean increasing permissiveness. When display breaks are enabled with \allowdisplaybreaks, the \\* command can be used to prohibit a pagebreak after a given line, as usual.
\setlist[enumerate]{font=\upshape,noitemsep, topsep=0pt} % while enumerating the numbering will be in up-shape. Default is italics. Also reduce item seperation space.
\setlist[itemize]{noitemsep, topsep=0pt}

\begin{document}
	
	\title[Necessary Conditions for $\Gamma_{E(3; 3; 1, 1, 1)}$, $\Gamma_{E(3; 2; 1, 2)}$ and $\mathcal{\bar{P}}$-Isometric Dilation]{Necessary Conditions for $\Gamma_{E(3; 3; 1, 1, 1)}$-Isometric Dilation, $\Gamma_{E(3; 2; 1, 2)}$-Isometric Dilation and $\mathcal{\bar{P}}$-Isometric Dilation}
	\author{Avijit Pal and Bhaskar Paul}
	\address[A. Pal]{Department of Mathematics, IIT Bhilai, 6th Lane Road, Jevra, Chhattisgarh 491002}
\email{A. Pal:avijit@iitbhilai.ac.in}

\address[B. Paul]{Department of Mathematics, IIT Bhilai, 6th Lane Road, Jevra, Chhattisgarh 491002}
\email{B. Paul:bhaskarpaul@iitbhilai.ac.in }



	\maketitle
	
	\begin{abstract}
A fundamental theorem of Sz.-Nagy states that a contraction $T$ on a Hilbert space can be dilated to an isometry $V.$ A more multivariable context of recent significance for these concepts involves substituting the unit disk with $\Gamma_{E(3; 3; 1, 1, 1)}, \Gamma_{E(3; 2; 1, 2)},$ and pentablock.	We demonstrate the necessary conditions for the existence of $\Gamma_{E(3; 3; 1, 1, 1)}$-isometric dilation, $\Gamma_{E(3; 2; 1, 2)}$-isometric dilation and pentablock-isometric dilation. We construct a class of $\Gamma_{E(3; 3; 1, 1, 1)}$-contractions and $\Gamma_{E(3; 2; 1, 2)}$-contractions that are always dilate	. We create an example of a $\Gamma_{E(3; 3; 1, 1, 1)}$-contraction that has a $\Gamma_{E(3; 3; 1, 1, 1)}$-isometric dilation such that  $[F_{7-i}^*, F_j] \ne [F_{7-j}^*, F_i] $ for some $i,j$ with $1\leq i ,j\leq 6,$ where $F_i$ and $F_{7-i}, 1\leq i \leq 6$ are the fundamental operators of $\Gamma_{E(3; 3; 1, 1, 1)}$-contraction $\textbf{T}=(T_1, \dots, T_7).$ We also produce an example of a $\Gamma_{E(3; 2; 1, 2)}$-contraction that has a $\Gamma_{E(3; 2; 1, 2)}$-isometric dilation by which $$[G^*_1, G_1] \neq [\tilde{G}^*_2, \tilde{G}_2]~{\rm{ and }}~[2G^*_2, 2G_2] \neq
 [2\tilde{G}^*_1, 2\tilde{G}_1],$$ where $G_1, 2G_2, 2\tilde{G}_1, \tilde{G}_2$ are the fundamental operators of $\textbf{S}$. As a result, the set of sufficient conditions for the existence of a $\Gamma_{E(3; 3; 1, 1, 1)}$-isometric dilation and $\Gamma_{E(3; 2; 1; 2)} $-isometric dilations presented in Theorem \ref{conddilation} and Theorem \ref{condilation1}, respectively, are not generally necessary.  We construct explicit $\Gamma_{E(3; 3; 1, 1, 1)} $-isometric, $\Gamma_{E(3; 2; 1; 2)} $-isometric dilations and $\mathcal{\bar{P}}$-isometric dilation of $\Gamma_{E(3; 3; 1, 1, 1)}$-contraction, $\Gamma_{E(3; 2; 1; 2)}$-contraction and $\mathcal{\bar{P}}$-contraction, respectively.  However, the question of whether a $\Gamma_{E(3; 3; 1, 1, 1)}$-isometric dilation, $\Gamma_{E(3; 2; 1, 2)}$-isometric dilation and $\mathcal{\bar{P}}$-isometric dilation for a  $\Gamma_{E(3; 3; 1, 1, 1)}$-contraction,  $\Gamma_{E(3; 2; 1, 2)}$-contraction, and $\mathcal{\bar{P}}$-contraction, respectively, remains unresolved.

\end{abstract}
	
	\section{Introduction and Motivation}\label{Sec 1}
Let $\mathbb C[z_1,\dots,z_n]$ represent the polynomial ring in $n$ variables over the field of complex numbers. Let $\Omega$ be a compact set in $\mathbb C ^m,$ and let $\mathcal A(\Omega)$ denote the algebra of holomorphic functions on an open set $U$ that contains $\Omega.$ Let  $\mathbf{T}=(T_1,\ldots,T_m)$ represent a commuting $m$-tuple of bounded operators defined on a Hilbert space $\mathcal H$ and $\sigma(\mathbf T)$ denotes the joint spectrum of the operator $\mathbf {T}.$  The mapping $\rho_{\mathbf T}:\mathcal A(\Omega)\rightarrow\mathcal B(\mathcal H)$ is defined as follows: $$1\to I~{\rm{and}}~z_i\to T_i ~{\rm{for}}~1\leq i\leq m.$$ It is evident that  $\rho_{\mathbf T}$ is a homomorphism. A compact set $\Omega\subset \mathbb C^m$ is defined as a spectral set for a  $m$-tuple of commuting bounded operators $\mathbf{T}=(T_1,\ldots,T_m)$ if $\sigma(\mathbf T)\subseteq \Omega$ and the homomorphism $\rho_{\mathbf T}:\mathcal A(\Omega)\rightarrow\mathcal B(\mathcal H)$ is contractive. A significant development for future research in non-self-adjoint operator theory is the Sz.-Nagy dilation theorem \cite{ pisier, paulsen}:  for a contraction $T\in \mathcal B(\mathcal H)$, there exists a larger Hilbert space $\mathcal K$ that contains $\mathcal H$ as a subspace, and a unitary operator $U$ acting on a Hilbert space $\mathcal K \supseteq \mathcal H$ with the property that $\mathcal K$ is the smallest closed reducing subspace for $U$ containing $\mathcal H$ such that $$P_\mathcal H\, U^n_{|\mathcal H}=T^n, ~{\rm{ for ~all}} ~n\in \mathbb N\cup \{0\}.$$ Schaffer constructed this type of unitary dilation for a contraction $T$. The spectral theorem for unitary operators demonstrates the validity of the von Neumann inequality: for any contraction $T\in \mathcal B(\mathcal H)$, 
$$\|p(T)\|\leq \|p\|_{\infty, \bar{\mathbb D}}:=\sup\{|p(z)|: |z|\leq1\} $$ holds for every polynomial $p.$ Let $\Omega$ be a compact subset of $\mathbb C ^m. $ Let $F=\left(\!(f_{ij})\!\right)$ be a matrix-valued polynomial defined on $\Omega.$ We call $\Omega$ a complete spectral set (complete $\Omega$-contraction) for $\mathbf T$ if the inequality $\|F(\mathbf T) \| \leq \|F\|_{\infty, \Omega}$ is satisfied for every $F\in \mathcal O(\Omega)\otimes \mathcal M_{k\times k}(\mathbb C), k\geq 1$. If $\Omega$ is a spectral set for a commuting $m$-tuple of operators $\mathbf{T},$ then it is a complete spectral set for $\mathbf{T},$ and we denote that the domain $\Omega$ has property $P.$
We define a $m$-tuple of commuting bounded operators $\mathbf{T}$ with $\Omega$ as a spectral set to possess a $\partial \Omega$ normal dilation if there exists a Hilbert space $\mathcal K$ that contains $\mathcal H$ as a subspace, along with a commuting $m$-tuple of normal operators $\mathbf{N}=(N_1,\ldots,N_m)$ on $\mathcal K$ with its spectrum contained in  $\partial \Omega$, satisfying the condition  $$P_{\mathcal H}F(\mathbf N)\mid_{\mathcal H}=F(\mathbf T) ~{\rm{for~ all~}} F\in \mathcal O(\Omega). $$ In 1969, Arveson \cite{A,AW} demonstrated that a commuting $m$-tuple of operators $\mathbf{T}$ admits a $\partial \Omega$ normal dilation if and only if $\Omega$ is a spectral set for $\mathbf{T}$ and $\mathbf{T}$ satisfies the property $P.$ In a single variable domain $\Omega \subset \mathbb C $, an annulus possesses the property $P$ \cite{agler}; however, this property does not hold for domains with connectivity $n \geq 2$ \cite{michel}.  In a higher-dimensional domain $\Omega$, the bi-disc possesses property $P$, as shown by Ando \cite{paulsen}. Furthermore, Agler and Young established normal dilation for a pair of commuting operators with the symmetrized bidisc as a spectral set \cite{JAgler, young}. However, the first counterexample in the multivariable context was given by Parrott \cite{paulsen}, which is for $\mathbb D^n$ when $n > 2$. G. Misra \cite{GM,sastry}, V. Paulsen \cite{vern}, and E. Ricard \cite{pisier} demonstrated that no ball in $\mathbb{C}^m$, with respect to some norm $\|\cdot\|_{\Omega}$ for $m \geq 3$, can have property $P$. It is further shown in \cite{cv} that if $B_1$ and $B_2$ are not simultaneously diagonalized through unitary, the set $\Omega_{\mathbf B}:= \{(z_1,z_2) :\|z_1 B_1 + z_2 B_2 \|_{\rm op} < 1\}$ fails to have property $P$, where $\mathbf B=(B_1, B_2)$ in $\mathbb C^2 \otimes \mathcal M_2(\mathbb C)$ with $B_1$ and $B_2$ are linearly independent.

Let $\mathcal M_{n\times n}(\mathbb{C})$ denote the set of all $n\times n$ complex matrices and  $E$ represent a linear subspace of $\mathcal M_{n\times n}(\mathbb{C}).$ The function $\mu_{E}: \mathcal M_{n\times n}(\mathbb{C}) \to [0,\infty)$ is defined as follows:
\begin{equation}\label{mu}
\mu_{E}(A):=\frac{1}{\inf\{\|X\|: \,\ \det(1-AX)=0,\,\, X\in E\}},\;\; A\in \mathcal M_{n\times n}(\mathbb{C})
	\end{equation}
with the understanding that $\mu_{E}(A):=0$ if $1-AX$ is  nonsingular for all $X\in E$ \cite{aj}.  We denote $\|\cdot\|$ as the operator norm. Let  $E(n;s;r_{1},\dots,r_{s})\subset \mathcal M_{n\times n}(\mathbb{C})$ be the vector subspace consisting of  block diagonal matrices, defined as follows:
\begin{equation}\label{ls}
    	E=E(n;s;r_{1},...,r_{s}):=\{\operatorname{diag}[z_{1}I_{r_{1}},....,z_{s}I_{r_{s}}]\in \mathcal M_{n\times n}(\mathbb{C}): z_{1},...,z_{s}\in \mathbb{C}\},
\end{equation}
 where $\sum_{i=1}^{s}r_i=n.$ We revisit the definition of $\Gamma_{E(3; 3; 1, 1, 1)}$, $\Gamma_{E(3; 2; 1, 2)}$ and $\Gamma_{E(2; 2; 1, 1)},$ $\mathcal{\bar{P}}$ \cite{Abouhajar,Young, Bharali,apal1}. The sets $\Gamma_{E{(2;2;1,1)}}$, $\mathcal{\bar{P}},$ $\Gamma_{E(3; 3; 1, 1, 1)}$ and $\Gamma_{E(3; 2; 1, 2)}$ are defined as 
\small{ \begin{equation*}
\begin{aligned}
\Gamma_{E{(2;2;1,1)}}:=\Big \{\textbf{x}=(x_1=a_{11}, x_2=a_{22}, x_3=a_{11}a_{22}-a_{12}a_{21}=\det A)\in \mathbb C^3: A\in \mathcal M_{2\times 2}(\mathbb C)~{\rm{and}}~\|A\|\leq 1\Big \},
\end{aligned}
\end{equation*}	}
 \begin{equation*}
\begin{aligned}
\mathcal {\bar{P}}=\Big \{\textbf{x}=(x_1=a_{21}, x_2=\tr(A), x_3=a_{11}a_{22}-a_{12}a_{21}=\det A)\in \mathbb C^3: A\in \mathcal M_{2\times 2}(\mathbb C)~{\rm{and}}~\|A\|\leq 1\Big \},
\end{aligned}
\end{equation*}	 
 \begin{equation*}
\begin{aligned}
\Gamma_{E{(3;3;1,1,1)}}:=\Big \{\textbf{x}=(x_1=a_{11}, x_2=a_{22}, x_3=a_{11}a_{22}-a_{12}a_{21}, x_4=a_{33}, x_5=a_{11}a_{33}-a_{13}a_{31},&\\ x_6=a_{22}a_{33}-a_{23}a_{32},x_7=\det A)\in \mathbb C^7: A\in \mathcal M_{3\times 3}(\mathbb C)~{\rm{and}}~\mu_{E(3;3;1,1,1)}(A)\leq 1\Big \}
\end{aligned}
\end{equation*}	
$${\rm{and}}$$  
\begin{equation*}
\begin{aligned}
\Gamma_{E(3;2;1,2)}:=\Big\{( x_1=a_{11},x_2=\det \left(\begin{smallmatrix} a_{11} & a_{12}\\
					a_{21} & a_{22}
				\end{smallmatrix}\right)+\det \left(\begin{smallmatrix}
					a_{11} & a_{13}\\
					a_{31} & a_{33}
				\end{smallmatrix}\right),x_3=\operatorname{det}A, y_1=a_{22}+a_{33}, &\\ y_2=\det  \left(\begin{smallmatrix}
					a_{22} & a_{23}\\
					a_{32} & a_{33}\end{smallmatrix})\right)\in \mathbb C^5
:A\in \mathcal M_{3\times 3}(\mathbb C)~{\rm{and}}~\mu_{E(3;2;1,2)}(A)\leq 1\Big\}.
\end{aligned}
\end{equation*}
The domains $\Gamma_{E(3; 2; 1, 2)}$, $\Gamma_{E(2; 2; 1, 1)}$  and $\mathcal {\bar{P}}$ are known  as as $\mu_{1,3}-$\textit{quotient}, tetrablock and pentablock, respectively \cite{Abouhajar, Young, Bharali}. 
\begin{defn}

Let $(A, B, P)$ be a commuting triple of bounded operators on a Hilbert space $\mathcal H$. We define $(A,B,P)$ as a tetrablock contraction if $\Gamma_{E(2;2;1,1)}$ is a spectral set for $(A,B,P)$. 
\end{defn}
The symmetrized bidisc and the tetrablock have drawn recent interest from complex analysts and operator theorists. Young's study on the symmetrized bidisc and the tetrablock, carried out with several co-authors \cite{Abouhajar, JAgler, young, ay, ay1, JAY, JANY}, has approached the topic from an operator-theoretic perspective.  Various authors studied the properties of $\Gamma_n$-isometries, $\Gamma_n$-unitaries, the Wold decomposition, and sufficient conditions for rational dilation of a $\Gamma_n$-contraction \cite{SS, A. Pal}. T. Bhattacharyya investigated the properties of  tetrablock isometries, tetrablock unitaries, the Wold decomposition for tetrablock, and sufficient conditions for rational dilation of a tetrablock-contraction \cite{Bhattacharyya}. H. Sau and J. Ball provided an example of tetrablock-contraction which has tetrablock-isometric dilation but fails to satisfy the sufficient conditions for rational dilation of a tetrablock-contraction which was given in \cite{Bhattacharyya}.  The similar results hold for the case of $\Gamma_n, n\geq3$ \cite{Mandal}. However, whether the tetrablock and $\Gamma_n, n\geq3,$ have the property $P$ remains unresolved.

Let  $$K=\{\textbf{x}=(x_1,\ldots,x_7)\in \Gamma_{E(3;3;1,1,1)} :x_1=\bar{x}_6x_7, x_3=\bar{x}_4x_7,x_5=\bar{x}_2x_7 ~{\rm{and}}~|x_7|=1\}$$ 
$$\rm{and}$$
\[ K_1 = \{x = (x_1, x_2, x_3, y_1, y_2) \in \Gamma_{E(3;2;1,2)} : x_1 = \overline{y}_2 x_3, x_2 = \overline{y}_1 x_3, |x_3| = 1 \}. \]

We begin with the following definitions that will be essential for our discussion.

\begin{defn}\label{def-1}
		\begin{enumerate}
			\item If $\Gamma_{E(3; 3; 1, 1, 1)}$ is a spectral set for $\textbf{T} = (T_1, \dots, T_7)$, then the $7$-tuple of commuting bounded operators $\textbf{T}$ defined on a  Hilbert space $\mathcal{H}$ is referred to as a \textit{$\Gamma_{E(3; 3; 1, 1, 1)}$-contraction}.
			
			\item Let $(S_1, S_2, S_3)$ and $(\tilde{S}_1, \tilde{S}_2)$ be tuples of commuting bounded operators defined on a Hilbert space $\mathcal{H}$ with $S_i\tilde{S}_j = \tilde{S}_jS_i$ for $1 \leqslant i \leqslant 3$ and $1 \leqslant j \leqslant 2$. We say that  $\textbf{S} = (S_1, S_2, S_3, \tilde{S}_1, \tilde{S}_2)$ is a $\Gamma_{E(3; 2; 1, 2)}$-contraction if $ \Gamma_{E(3; 2; 1, 2)}$ is a spectral set for $\textbf{S}$.
			
\item A commuting $7$-tuple of normal operators $\textbf{N} = (N_1, \dots, N_7)$ defined on a Hilbert space $\mathcal{H}$ is  a \textit{$\Gamma_{E(3; 3; 1, 1, 1)}$-unitary} if the Taylor joint spectrum $\sigma(\textbf{N})$ is contained in the set $K$.  		
			\item A commuting $5$-tuple of normal operators $\textbf{M} = (M_1, M_2, M_3, \tilde{M}_1, \tilde{M}_2)$ on a Hilbert space $\mathcal{H}$ is referred as a \textit{$\Gamma_{E(3; 2; 1, 2)}$-unitary} if the Taylor joint spectrum $\sigma(\textbf{M})$ is contained in $K_1.$ 
					
\item A  $\Gamma_{E(3; 3; 1, 1, 1)}$-isometry (respectively, $\Gamma_{E(3; 2; 1, 2)}$-isometry) is defined as the restriction of a $\Gamma_{E(3; 3; 1, 1, 1)}$-unitary (respectively, $\Gamma_{E(3; 2; 1, 2)}$-unitary)  to a joint invariant subspace. In other words, a $\Gamma_{E(3; 3; 1, 1, 1)}$-isometry ( respectively, $\Gamma_{E(3; 2; 1, 2)}$-isometry) is a $7$-tuple (respectively, $5$-tuple) of commuting bounded operators that possesses simultaneous extension to a \textit{$\Gamma_{E(3; 3; 1, 1, 1)}$-unitary} (respectively, \textit{$\Gamma_{E(3; 2; 1, 2)}$-unitary}). It is important to observe that a $\Gamma_{E(3; 3; 1, 1, 1)}$-isometry (respectively, $\Gamma_{E(3; 2; 1, 2)}$-isometry ) $\textbf{V}=(V_1\dots,V_7)$ (respectively,  $\textbf{W}=(W_1,W_2,W_3,\tilde{W}_1,\tilde{W}_2)$) consists of commuting subnormal operators with $V_7$ (respectively, $W_3$)  is an isometry.

\item   We say that $\textbf{V}$ (respectively, $\textbf{W}$)   is a pure $\Gamma_{E(3; 3; 1, 1, 1)}$-isometry (respectively, pure $\Gamma_{E(3; 2; 1, 2)}$-isometry) if $V_7$ (respectively, $W_3$) is a  pure isometry,  that is, a shift of some multiplicity.		
\end{enumerate}
	\end{defn}	
	
Let
	\begin{equation}\label{K_0}
		\begin{aligned}
			K_0 &=
			\left\{(x_1, x_2, x_3) \in \mathbb{C}^3 : |x_2| \leqslant 2, |x_3| = 1, x_2 = \overline{x}_2x_3 \,\, \text{and} \,\, |x_1| = \sqrt{1 - \frac{1}{4}|x_2|^2} \right\}.
		\end{aligned}
	\end{equation}
The following theorem characterizes the distinguished boundary of the pentablock \cite{Young}.
\begin{thm}[Theorem 8.4, \cite{Young}]\label{distinguished boundary pentablock}
		For $x \in \mathbb{C}^3$ the following are equivalent:
		\begin{enumerate}
			\item $x \in K_0$,
			
			\item $x$ is a peak point of $\overline{\mathcal{P}}$,
			
			\item $x \in b\mathcal{\bar{P}}$, the distinguished boundary of $\mathcal {P}$.
		\end{enumerate}
	\end{thm}
We recall the definition of pentablock contraction, pentalblock unitary, and pentalblock isometry from \cite{Jindal}.
\begin{defn}\label{Defn 2}
		Let $\textbf{P}=(P_1, P_2, P_3)$ be a commuting triple of bounded operators on a Hilbert space $\mathcal{H}$. We call it
		\begin{enumerate}
			\item If $\overline{\mathcal{P}}$ is a spectral set for $\textbf{P}=(P_1, P_2, P_3)$, then a commuting triple of bounded operators $\textbf{P}$ on a Hilbert space $\mathcal{H}$ is said to be a pentablock contraction. 			
			\item A commuting triple of normal operators $\textbf{P}=(P_1, P_2, P_3)$ on a Hilbert space $\mathcal{H}$ is called a pentablock unitary ($\mathcal{\bar{P}}$-unitary) if  the Taylor joint spectrum $\sigma(\textbf{P})$ is contained in $b\mathcal{P}$.			\item A pentablock isometry ($\mathcal{\bar{P}}$-isometry) is defined as the restriction of a pentablock unitary to a joint invariant subspace. 
\item We define a pentablock isometry as pure if $P_3$ is a pure isometry, that is, a shift of some multiplicity.			

		\end{enumerate}
	\end{defn}
	
Let  $\mathbb T$ be the unit circle. We shall use some spaces of vector-valued and operator-valued functions. Let $\mathcal E $ be a separable Hilbert space. Let $\mathcal B(\mathcal E)$ be the space of all bounded operators on $\mathcal E$ with respect to the operator norm. Let $H^2(\mathcal E)$ denote the standard Hardy space of analytic $\mathcal E$-valued functions defined on the unit disk  $\mathbb D,$ whereas $ L^2(\mathcal E)$ represents the Hilbert space of square-integrable $\mathcal E$-valued functions on the unit circle $\mathbb T,$ equipped with their natural inner products. The space $H^{\infty}(\mathcal B(\mathcal E))$ consists of bounded analytic $\mathcal B(\mathcal E)$-valued functions defined on $\mathbb D$, while  $L^{\infty}(\mathcal B(\mathcal E))$ represents the space of bounded measurable functions with values in $\mathcal B(\mathcal E)$ defined on $\mathbb T$. Both spaces have the appropriate version of the supremum norm. For $\phi \in L^{\infty}(\mathcal B(\mathcal E)),$ the Toeplitz operator corresponding to  the symbol  $\phi$ is denoted by $T_{\phi}$ and is defined as follows: 
$$T_{\phi}f=P_{+}(\phi f), f \in H^2(\mathcal E),$$ where $P_{+} : L^2(\mathcal E) \to H^2(\mathcal E)$ is the orthogonal projecton.  Specifically, $M_z$ represents the unilateral shift operator on $H^2(\mathcal E)$ (we denote the identity function on $\mathbb T$ by $z$) and $M_{\bar{z}}$ denotes the backward shift operator on $H^2(\mathcal E).$

In section $2$, we prove the necessary conditions for the existence of $\Gamma_{E(3; 3; 1, 1, 1)}$-isometric dilation and $\Gamma_{E(3; 2; 1, 2)}$-isometric dilation. We construct a class of $\Gamma_{E(3; 3; 1, 1, 1)}$-contractions that are always dilate, specifically those of the form $\textbf{T} = (T_1, T_2, T_1T_2, T_1T_2, T_1^2T_2, T_1T_2^2, T_1^2T_2^2),$ where $(T_1, T_2)$ denotes a pair of contractions. Furthermore, we discuss a class of $\Gamma_{E(3; 2; 1, 2)}$-contractions that always dilate, particularly those of the form $\textbf{S} = (S_1, S_1S_2+S_1^2S_2, S_1^2S_2^2, S_2+S_1S_2, S_1S_2^2),$ where $(S_1, S_2)$ is a pair of contractions. We establish a necessary condition for the existence of a $\mathcal{\bar{P}}$-isometric dilation in section $3$. In section $4$, we produce an example of a $\Gamma_{E(3; 3; 1, 1, 1)}$-contraction that has a $\Gamma_{E(3; 3; 1, 1, 1)}$-isometric dilation such that  $[F_{7-i}^*, F_j] \ne [F_{7-j}^*, F_i] $ for some $i,j$ with $1\leq i ,j\leq 6.$ In conclusion, we assert that the set of sufficient conditions for the existence of a $\Gamma_{E(3; 3; 1, 1, 1)}$-isometric dilation presented in Theorem \ref{conddilation} are generally not necessary, even when the $\Gamma_{E(3; 3; 1, 1, 1)}$-contraction $\textbf{T}=(T_1,\dots, T_7)$ has a special form, where $T_7$ is a partial isometry on $\mathcal H.$  Furthermore, we also provide an example of a $\Gamma_{E(3; 2; 1, 2)}$-contraction that has a $\Gamma_{E(3; 2; 1, 2)}$-isometric dilation by which one of the conditions outlined in the Proposition \ref{Prop Partial 1} is not satisfied. In summary, we conclude that the set of sufficient conditions for the existence of a $\Gamma_{E(3; 2; 1, 2)}$-isometric dilation described in Theorem \ref{condilation1} are not generally necessary, even when the $\Gamma_{E(3; 2; 1, 2)}$-contraction $\textbf{S} = (S_1, S_2, S_3, \tilde{S}_1, \tilde{S}_2)$ has a special form, particularly where $S_3$ is a partial isometry on $\mathcal H.$ In section $5$, we construct explicit $\Gamma_{E(3; 3; 1, 1, 1)} $-isometric and $\Gamma_{E(3; 2; 1; 2)} $-isometric dilations of $\Gamma_{E(3; 3; 1, 1, 1)}$-contraction and $\Gamma_{E(3; 2; 1; 2)}$-contraction, respectively. We construct a family of $\mathcal{\bar{P}}$-contractions that have $\mathcal{\bar{P}}$-isometric dilation in  section $6.$ However, the question of whether a $\Gamma_{E(3; 3; 1, 1, 1)}$-isometric dilation, $\Gamma_{E(3; 2; 1, 2)}$-isometric dilation and $\mathcal{\bar{P}}$-isometric dilation for a  $\Gamma_{E(3; 3; 1, 1, 1)}$-contraction,  $\Gamma_{E(3; 2; 1, 2)}$-contraction, and $\mathcal{\bar{P}}$-contraction, respectively, remains unresolved.  

	\section{$\Gamma_{E(3; 3; 1, 1, 1)}$-Isometric Dilation and $\Gamma_{E(3; 2; 1, 2)}$-Isometric Dilation : Necessary and Sufficient Conditions}\label{Sec 2}
We revisit the definitions for the terms spectrum, spectral radius, and numerical radius of an operator. Let $\sigma(T)$ denote the spectrum of $T$, defined as $$\sigma(T)=\{\lambda \in \mathbb{C} \mid T-\lambda I~{\rm{ is~not~invertible}}\}. $$ Additionally, the numerical radius of a bounded operator \( T \) on a Hilbert space \( \mathcal{H} \) is represented as $$\omega(T)=\sup\{|\langle Tx,x\rangle|:\| x \|=1\}. $$ A direct computation demonstrates that $r(T)\leq \omega(T)\leq \|T\|$ for a bounded operator $T$, where the spectral radius is defined as $$r(T)= \sup_{ \lambda \in \sigma(T)}|\lambda|. $$	Let \( T \) be a contraction on a Hilbert space \( \mathcal{H} \). The defect operator associated with \( T \) is defined as \( D_{T} = (I - T^*T)^{\frac{1}{2}} \). The closure of the range of \( D_{T} \) is denoted by \( \mathcal{D}_{T} \). Halmos initially observed that if $U=\left(\begin{smallmatrix} T & D_{T^*} \\ D_T & -T^*\end{smallmatrix}
\right),$ then $T=P_{\mathcal H}U_{|_{\mathcal H}}. $ An operator satisfying the criterion above can be referred to as a $1$-dilation. Let $\mathcal K=\underbrace{\mathcal H\oplus\cdots \oplus\mathcal H}_\text{$N+1$ times}$ and consider the operator matrix of size $(N+1)\times(N+1)$ defined as
\small{\begin{equation}\label{U}
U=\left(\begin{smallmatrix}
T & 0 & 0 & \cdots & 0 & D_{T^*}\\D_{T} & 0 &0 & \cdots & 0 &-T^*\\0 & I &0 & \cdots & 0 & 0\\0 & 0 &I & \cdots & 0 &0\\\vdots & \vdots & \vdots & \cdots & \vdots & \vdots\\ 0 & 0 &0 & \cdots & I & 0
\end{smallmatrix}\right)
\end{equation}}
\newline ${\rm{Egerv\grave{a}ry}}$ proved that $U$ is a unitary operator on $\mathcal K$ and satisfies the following conditions: \begin{equation}\label{Uk}U^k=\left(\begin{smallmatrix} T^k & 0\\ 0 & *\end{smallmatrix}
\right),~k=1,\cdots, N.
\end{equation} By identifying $\mathcal H$ with the first summand of $\mathcal K,$ for every polynomial $p$ of degree at most $N,$ it follows that $p(T)=P_{\mathcal H}P(U)_{|_{\mathcal H}}.$ A dilation of this type is referred to as $N$-dilation.
An operator $U\in \mathcal B(\mathcal K)$ is called a power dilation of $T\in \mathcal B(\mathcal H)$ if $\mathcal H$ is a subspace of $\mathcal K$ and if for all $k=0,1,2,\dots, T^k=P_{\mathcal H}U^k_{|_{\mathcal H}}$.
\begin{thm}[Sz.-Nagy's isometric dilation, \cite{shalit}] Let $T$ be a contraction acting on a Hilbert space $\mathcal H.$ Then there exists a Hilbert space $\mathcal K$ that contains $\mathcal H$ as a subspace and an isometry $V$ on $\mathcal K$ such that
$$T^*=V^*_{|_{\mathcal H}}$$ and, in particular, $V$ serves as the power dilation of $T.$ Moreover, $\mathcal K$ can be chosen as minimal, indicating that the minimal invariant subspace for $V$ that includes $\mathcal H$ is $\mathcal K$.
\end{thm}
The minimal isometric dilation is indeed a co-extension, and it has been demonstrated that co-extension is always a power dilation. However, the converse is not true. We now define the $\Gamma_{E(3; 3; 1, 1, 1)}$-isometric dilation of the $\Gamma_{E(3; 3; 1, 1, 1)}$-contraction and the $\Gamma_{E(3; 2; 1, 2)}$-isometric dilation of the $\Gamma_{E(3; 2; 1, 2)}$-contraction.

\begin{defn}\label{isometric dilation2}
A commuting $7$-tuple of operators $(V_1,\ldots,V_{7})$ acting on a Hilbert space $\mathcal K \supseteq \mathcal H$ is referred to as a $\Gamma_{E(3; 3; 1, 1, 1)}$ -isometric dilation of a $\Gamma_{E(3; 3; 1, 1, 1)}$-contraction $(T_1,\ldots,T_7)$ acting on a Hilbert space $\mathcal H$ possesses the following properties:
\begin{itemize}
\item  $(V_1,\ldots,V_{7})$ is $\Gamma_{E(3; 3; 1, 1, 1)}$-isometry;
\item $V_i^{*}|_{\mathcal H}=T_i^{*}$ for all $1\leq i \leq 7.$
\end{itemize}

\end{defn}	
It follows from the above definition that $(V_1,\ldots,V_{7}) $ is a $\Gamma_{E(3; 3; 1, 1, 1)} $-isometric dilation of a $\Gamma_{E(3; 3; 1, 1, 1)} $-contraction $(T_1,\ldots,T_7) $ is equivalent to stating that $(V_1^*,\ldots,V_{7}^*)$  is a $\Gamma_{E(3; 3; 1, 1, 1)} $-co-isometric extension of $(T_1^*,\ldots,T_7^*)$. Moreover, we call the dilation as minimal if $$\mathcal K_{0}=\overline{{\rm{span}}}\{ V_{7}^{n}h:h\in\mathcal H ~{\rm{and}}~n \in \mathbb N\cup\{0\}\}.$$

The operator functions $\rho_{G_{E(2; 1; 2)}}$ and $\rho_{G_{E(2; 2; 1,1)}}$ for the symmetrized bidisc and tetrablock are defined as follows:
$$\rho_{G_{E(2; 1; 2)}} (S,P)=2 (I-P^*P)-(S-S^*P)-(S^*-P^*S) $$ and 
$$\rho_{G_{E(2; 2; 1,1)}} (T_1,T_2,T_3)=(I-T_3^*T_3) -(T_2^*T_2-T_1^*T_1) -2\Re {T_2-T_1^*T_3}, $$ where $P,T_3$ are contractions and $S,P$ and $T_1,T_2,T_3$ are commuting bounded operators defined on Hilbert spaces $\mathcal H_1$ and $\mathcal H_2$, respectively. We review the definition of tetrablock contraction as stated in \cite{Bhattacharyya}.
\begin{defn}\label{fundamental}
		Let $(T_1, \dots, T_7)$ be a $7$-tuple of commuting contractions on a Hilbert space $\mathcal{H}. $ The equations 

		\begin{equation}\label{Fundamental 1}
\begin{aligned}
&T_i - T^*_{7-i} T_7 = D_{T_7}F_iD_{T_7}, 1\leq i \leq 6,
\end{aligned}
\end{equation}
where $F_i\in \mathcal{B}(\mathcal{D}_{T_7}),$ are referred to as the  fundamental equations for $(T_1, \dots, T_7)$.
			\end{defn}
For any $z\in \mathbb C$, we  introduce the operators $S^{(i)}_z = T_i + zT_{7-i}$ for $1\leq i \leq 6$ and $P_z = zT_7$.
\begin{thm}[ Theorem $2.4$, \cite{apal3}]\label{fundam}
		Let $\textbf{T} = (T_1, \dots, T_7)$ be a commuting $7$-tuple of bounded operators acting on a Hilbert space $\mathcal{H}$. Then in the following $(1) \Rightarrow (2) \Rightarrow (3) \Rightarrow (4) \Rightarrow (5):$
		\begin{enumerate}
			\item $\mathbf{T} = (T_1, \dots, T_7)$ be a $\Gamma_{E(3; 3; 1, 1, 1)}$-contraction.
			
			\item $(T_i, T_{7-i}, T_7)$ is a $\Gamma_{E(2; 2; 1, 1)}$-contraction for $1 \leq i \leq 6$.
			
			\item For $1\leq i \leq  6$ and $z \in \mathbb{T}$,
			\begin{equation*}
				\begin{aligned}
					&\rho_{G_{E(2; 2; 1, 1)}}(T_i, zT_{7-i}, zT_7) \geqslant 0, ~\text{and}~ \rho_{G_{E(2; 2; 1, 1)}}(T_{7-i}, zT_i, zT_7) \geqslant 0.
				\end{aligned}
			\end{equation*}
			and the spectral radius of $S^{(i)}_z$ is not bigger than $2$, for $1 \leq i \leq 6.$
			
			\item The pair $(S^{(i)}_z, P_z),1\leq i \leq 6,$ is a $\Gamma_{E(2; 1; 2)}$-contraction for every $z \in \mathbb{T}$.
			
			\item The fundamental equations in \eqref{Fundamental 1} have unique solutions $F_i$ and $F_{7-i}$ in $\mathcal{B}(\mathcal{D}_{T_7})$ for $1\leq i \leq 6.$  Moreover, the operator $F_i + zF_{7-i}, 1\leq i \leq 6,$ has numerical radius not bigger than $1$ for every $z \in \mathbb{T}$.
		\end{enumerate}
	\end{thm}	
The following theorem [Theorem $4.5$, \cite{apal3}] provides the sufficient conditions for the existence of  $\Gamma_{E(3; 3; 1, 1, 1)} $-isometric dilation under the assumption that $\mathbf{T} = (T_1, \dots, T_7)$ is a $\Gamma_{E(3; 3; 1, 1, 1)}$-contraction, with its fundamental operators $F_i$ and $F_{7-i},$ for $1\leq i \leq 6,$ which satisfy the following conditions:
\begin{equation}
[F_i,F_j]=0~~{\rm{and}}~~[F_{7-i}^*,F_j]=[F_{7-j}^*,F_i], 1\leq i,j \leq 6.
\end{equation}

\begin{thm}[Conditional Dilation of $\Gamma_{E(3; 3; 1, 1, 1)}$-Contraction]\label{conddilation}
		Let $\mathbf{T} = (T_1, \dots, T_7) $ be a $\Gamma_{E(3; 3; 1, 1, 1)} $-contraction define on a Hilbert space $\mathcal H$ with the fundamental operator $F_i$ and $F_{7-i},$ for $1\le i \leq 6,$ which satisfy the following conditions: \begin{enumerate}
			\item[$(i)$] $[F_i, F_j] = 0, 1\leq i, j \leq 6$;
			
			\item[$(ii)$] $[F^*_{7-i}, F_j] = [F^*_{7-j}, F_i], 1\leq i, j \leq  6$.
		\end{enumerate} Let 
		\[\mathcal{K} = \mathcal{H} \oplus \mathcal{D}_{T_7} \oplus \mathcal{D}_{T_7} \oplus \dots = \mathcal{H} \oplus l^2(\mathcal{D}_{T_7}).\] Let  $\mathbf{V}=(V_1, \dots, V_7)$ be a $7$-tuple of operators defined on $\mathcal K$
by
		\begin{equation}\label{v7}
			\begin{aligned}
				&V_i =
				\begin{bmatrix}
					T_i & 0 & 0 & \dots\\
					F^*_{7-i}D_{T_7} & F_i & 0 & \dots\\
					0 & F^*_{7-i} & F_i & \dots\\
					0 & 0 & F^*_{7-i} & \dots\\
					\vdots & \vdots & \vdots & \ddots
				\end{bmatrix}, 1\leq i\leq 6,~{\rm{and}}~~
				V_7 =
				\begin{bmatrix}
					T_7 & 0 & 0 & \dots\\
					D_{T_7} & 0 & 0 & \dots\\
					0 & I & 0 & \dots\\
					0 & 0 & I & \dots\\
					\vdots & \vdots & \vdots & \ddots
				\end{bmatrix}.
			\end{aligned}
		\end{equation}
Then we have the following:
\begin{enumerate}

\item $\mathbf{V}$ is a minimal $\Gamma_{E(3; 3; 1, 1, 1)} $-isometric dilation of $\mathbf{T}$.

\item If there exists a $\Gamma_{E(3; 3; 1, 1, 1)} $-isometric dilation $\textbf{W}= (W_1, \dots, W_7)$ of $\textbf{T}$ such that $W_7$ is a minimal isometric dilation of $T_7$, then $\textbf{W}$ is unitarily equivalent to $\textbf{V}$. Furthermore, the above conditions $(i)$ and $(ii)$ are also valid.
\end{enumerate}
	\end{thm}
We will establish the necessary conditions for the existence of $\Gamma_{E(3; 3; 1, 1, 1)}$-isometric dilation. 
	
	\begin{thm}\label{Necessary Conditions 1}
		Let $\textbf{T} = (T_1, \dots, T_7)$ be a $\Gamma_{E(3; 3; 1, 1, 1)}$-contraction on a Hilbert space $\mathcal{H}$ with fundamental operators  $F_i,1\leq i \leq 6.$ Then each of the following conditions is necessary for $\textbf{T}$ to have a $\Gamma_{E(3; 3; 1, 1, 1)}$-isometric dilation:
		\begin{enumerate}
			\item The $6$-tuple of operator $(F_1, \dots, F_6)$ has a joint dilation to a $6$-tuple of commuting subnormal operator $(\tilde{F}_1, \dots, \tilde{F}_6)$, that is, there exists an isometric embedding $\Theta$ of $\mathcal{D}_{T_7}$ into a larger Hilbert space $\mathcal{E}$ so that $F_j = \Theta^*\tilde{F}_j\Theta$ for $1 \leqslant j \leqslant 6,$ where $(\tilde{F}_1, \dots, \tilde{F}_6)$ can be extended to a  $6$-tuple of commuting normal operators $(N_1, \dots, N_6)$ with Taylor joint spectrum contained in the union of the $6$-tori $$\{(z_1, \dots, z_6) : |z_i| = |z_{7-i}| \leqslant 1 \,\, \text{for} \,\, 1 \leqslant i \leqslant 6\}.$$
			
			\item $(F^*_iD_{T_7}T_i - F^*_{7-i}D_{T_7}T_{7-i})|_{Ker D_{T_7}} = 0$ for $1 \leqslant i \leqslant 6$.
			
			\item $(F^*_iF^*_{7-i} - F^*_{7-i}F^*_i)D_{T_7}T_7|_{Ker D_{T_7}} = 0$ for $1 \leqslant i \leqslant 6$.
			
					\end{enumerate}
	\end{thm}
	
	\begin{proof}
		Suppose that $\textbf{V} = (V_1, \dots, V_7)$ is a $\Gamma_{E(3; 3; 1, 1, 1)}$-isometric dilation of $\textbf{T}.$  It is important to note that $\Gamma_{E(3; 3; 1, 1, 1)}$ is a polynomially convex [Theorem $3.4$, \cite{apal1}]. Therefore, it is sufficient to work with polynomials instead of the entire algebra $\mathcal O(\Gamma_{E(3; 3; 1, 1, 1)}),$ and we can assume, without loss of generality, that	
\[\mathcal{K} = \overline{\rm{span}}\{V^{n_1}_1 \dots V^{n_7}_7h : h \in \mathcal{H}, n_1, \dots, n_7 \in \mathbb{N} \cup \{0\}\}.\]
		By definition, we have
		\[(V^*_1, \dots, V^*_7)|_{\mathcal{H}} = (T^*_1, \dots, T^*_7).\]
Let the $2\times 2$ block operator matrix of $V_i$ be of the form
		\begin{equation}\label{E 1}
\begin{aligned}
V_i&=
\begin{pmatrix}
T_i& 0\\
C_i& \tilde{F}_i
\end{pmatrix} \,\,\text{for}\,\, 1 \leqslant i \leqslant 7,
\end{aligned}
\end{equation}
with respect to the decomposition $\mathcal{K} = \mathcal{H} \oplus (\mathcal{K} \ominus \mathcal{H})$ of $\mathcal{K}$.  As $V_7$ is an isometry, by using \eqref{E 1}, we deduce that
		\begin{equation}\label{E 2}
			\begin{aligned}
				T^*_7T_7 + C^*_7C_7 = I_{\mathcal{H}},  \tilde{F}^*_7\tilde{F}_7 = I_{\mathcal{K} \ominus \mathcal{H}}.
			\end{aligned}
		\end{equation}
It implies from \eqref{E 2} that there exists an isometry $\Theta: \mathcal{D}_{T_7} \to \mathcal{K} \ominus \mathcal{H}$ such that
		\begin{equation}\label{E 3}
			\begin{aligned}
				\Theta D_{T_7} &= C_7.
			\end{aligned}
		\end{equation}
As $\textbf{V} = (V_1, \dots, V_7)$ is a $\Gamma_{E(3; 3; 1, 1, 1)}$-isometric, it follows from [Theorem $4.4$, \cite{apal2}] that  $V_i = V^*_{7-i}V_7$ for $1\leq i \leq 6.$ Thus, we have for $1\leq i \leq 6$
\begin{equation}\label{E11}
\begin{aligned}
\begin{pmatrix}
T_i& 0\\
C_i& \tilde{F}_i
\end{pmatrix}&=\begin{pmatrix}
T_{7-i}^*& C_{7-i}^*\\
0& \tilde{F}_{7-i}^*
\end{pmatrix}\begin{pmatrix}
T_7& 0\\
C_7& \tilde{F}_7
\end{pmatrix}\\&=\begin{pmatrix}
 T^*_{7-i}T_7+C^*_{7-i}C_7& C^*_{7-i}\tilde{F}_7\\
 \tilde{F}^*_{7-i}C_7& \tilde{F}^*_{7-i}\tilde{F}_7
\end{pmatrix}
\end{aligned}
\end{equation}		
From \eqref{E11}, we get		
	\begin{equation}\label{E 4}
			\begin{aligned}
				T_i - T^*_{7-i}T_7 = C^*_{7-i}C_7, \,\, C^*_{7-i}\tilde{F}_7 = 0, \,\, C_i = \tilde{F}^*_{7-i}C_7, \,\, \text{and} \,\, \tilde{F}_i = \tilde{F}^*_{7-i}\tilde{F}_7.
			\end{aligned}
		\end{equation}
From \eqref{Fundamental 1} and \eqref{E 4}, we deduce that 
		\begin{equation}\label{E 5}
			\begin{aligned}
				D_{T_7}F_iD_{T_7}
				&= T_i - T^*_{7-i}T_7 = C^*_{7-i}C_7 = C^*_7\tilde{F}_iC_7 = D_{T_7}\Theta^*\tilde{F}_i\Theta D_{T_7}.
			\end{aligned}
		\end{equation}
		By the uniqueness of the fundamental operators $F_i,1\leq i \leq 6$, we conclude that  
		\begin{equation}\label{E 6}
			\begin{aligned}
				F_i &= \Theta^*\tilde{F}_i\Theta ~\text{for}~ 1 \leqslant i \leqslant 6.
			\end{aligned}
		\end{equation}
It yields from \eqref{E 2} and \eqref{E 4} that
		\begin{equation}\label{E 7}
			\begin{aligned}
				\tilde{F}_i = \tilde{F}^*_{7-i}\tilde{F}_7 \,\,\text{and}\,\, \tilde{F}^*_7\tilde{F}_7 = I_{\mathcal{K} \ominus \mathcal{H}}  ~\text{for}~ 1 \leqslant i \leqslant 6.
\end{aligned}
		\end{equation}
		Since $\textbf{V}$ is a $\Gamma_{E(3; 3; 1, 1, 1)}$-isometry and $(\tilde{F}_1, \dots, \tilde{F}_7)=(V_1, \dots, V_7)_{|_{\mathcal{K} \ominus \mathcal{H}}},$ it implies that $(\tilde{F}_1, \dots, \tilde{F}_7)$ is a $\Gamma_{E(3; 3; 1, 1, 1)}$-contraction. As $(\tilde{F}_1, \dots, \tilde{F}_7)$ is a $\Gamma_{E(3; 3; 1, 1, 1)}$-contraction, we conclude from \eqref{E 7} that $\tilde{\textbf{F}} = (\tilde{F}_1, \dots, \tilde{F}_7)$ is a $\Gamma_{E(3; 3; 1, 1, 1)}$-isometry, and so by definition of $\Gamma_{E(3; 3; 1, 1, 1)} $-isometry, $\tilde{\textbf{F}}$ has a $\Gamma_{E(3; 3; 1, 1, 1)}$-unitary extension $\textbf{N} = (N_1, \dots, N_7)$ on a larger Hilbert space. Since $\textbf{N} = (N_1, \dots, N_7)$ is  $\Gamma_{E(3; 3; 1, 1, 1)}$-unitary, it follows from the definition of  $\Gamma_{E(3; 3; 1, 1, 1)}$-unitary that  the Taylor joint spectrum $\sigma(\textbf{N})$ of $\textbf{N}$ is contained in $K$ and $N_1, \dots, N_7$ are commuting normal operators.  By ignoring the $7$th co-ordinate, we conclude that the Taylor joint spectrum of  $\sigma(N_1, \dots, N_6)$ is contained in the union of $6$-tori $\{(z_1, \dots, z_6) : |z_i| = |z_{7-i}| \leqslant 1 \,\, \text{for} \,\, 1 \leqslant i \leqslant 6\}$, and part $(1)$ follows.
			
As $V_iV_{7-i} = V_{7-i}V_i$ for $ 1\leq i \leq 6,$ we see that 
		\begin{equation}\label{E 8}
			\begin{aligned}
				&C_iT_{7-i} + \tilde{F}_iC_{7-i}
				= C_{7-i}T_i + \tilde{F}_{7-i}C_i.\\
				\end{aligned}
		\end{equation}
It follows from \eqref{E 4} and \eqref{E 8} that for $1\leq i \leq 6,$
\begin{equation}\label{E12}
\begin{aligned}
C_{7-i}T_i-C_iT_{7-i}&= \tilde{F}^*_{i}C_7T_i- \tilde{F}^*_{7-i}C_7T_{7-i}\\&= \tilde{F}^*_{i}\Theta D_{T_7}T_i- \tilde{F}^*_{7-i}\Theta D_{T_7}T_{7-i}
\end{aligned}
\end{equation}
$${\rm{and}}$$
\begin{equation}\label{E13}
\begin{aligned}
\tilde{F}_iC_{7-i}-\tilde{F}_{7-i}C_{i}&= \tilde{F}_i\tilde{F}^*_{i}C_7-\tilde{F}_{7-i} \tilde{F}^*_{7-i}C_7\\&= ( \tilde{F}_i\tilde{F}^*_{i}-\tilde{F}_{7-i} \tilde{F}^*_{7-i})\Theta D_{T_7}.
\end{aligned}
\end{equation}		
From \eqref{E 8},\eqref{E12} and \eqref{E13}, we have
\begin{equation}\label{E14}
\tilde{F}^*_{i}\Theta D_{T_7}T_i- \tilde{F}^*_{7-i}\Theta D_{T_7}T_{7-i}=( \tilde{F}_i\tilde{F}^*_{i}-\tilde{F}_{7-i} \tilde{F}^*_{7-i})\Theta D_{T_7}.
\end{equation}
By multiplying $\Theta^*$ on the left side of \eqref{E14} and using \eqref{E 6}, we observe that
		\begin{equation}\label{E 9}
			\begin{aligned}
				F^*_iD_{T_7}T_i - F^*_{7-i}D_{T_7}T_{7-i}
				&= \Theta^*\tilde{F}^*_{i}\Theta D_{T_7}T_i- \Theta^*\tilde{F}^*_{7-i}\Theta D_{T_7}T_{7-i}\\&=\Theta^*( \tilde{F}_i\tilde{F}^*_{i}-\tilde{F}_{7-i} \tilde{F}^*_{7-i})\Theta D_{T_7}.
\end{aligned}
		\end{equation}
From \eqref{E 9}, we deduce that 
		\[(F^*_iD_{T_7}T_i - F^*_{7-i}D_{T_7}T_{7-i})|_{Ker D_{T_7}} = 0,\] part $(2)$ follows.
			
 By [Lemma 2.7, \cite{apal2}] and \eqref{E 9}, we have
		\begin{equation}\label{E 10}
			\begin{aligned}
				\Theta^*(\tilde{F}_i\tilde{F}^*_i - \tilde{F}_{7-i}\tilde{F}^*_{7-i})\Theta D_{T_7}
				&= F^*_iD_{T_7}T_i - F^*_{7-i}D_{T_7}T_{7-i}\\
				&= F^*_i(F_iD_{T_7} + F^*_{7-i}D_{T_7}T_7) - F^*_{7-i}(F_{7-i}D_{T_7} + F^*_iD_{T_7}T_7)\\
				&= (F^*_iF_i - F^*_{7-i}F_{7-i})D_{T_7} - (F^*_iF^*_{7-i} - F^*_{7-i}F^*_i)D_{T_7}T_7.
				\end{aligned}
		\end{equation}
It follows from \eqref{E 10} that
		\[(F^*_iF^*_{7-i} - F^*_{7-i}F^*_i)D_{T_7}T_7|_{Ker \mathcal{D}_{T_7}} = 0.\]
This completes the proof.
	\end{proof}
We discuss a class of $\Gamma_{E(3; 3; 1, 1, 1)}$-contractions that are always dilate, specifically those of the form $\textbf{T} = (T_1, T_2, T_1T_2, T_1T_2, T_1^2T_2, T_1T_2^2, T_1^2T_2^2),$ where $(T_1,T_2)$ denotes a pair of contractions.
\begin{thm}\label{Thm Example 1}
		Let $(T_1, T_2)$ be a pair of commuting contractions on a Hilbert space $\mathcal{H}$. Then the $7$-tuple of operators $\textbf{T} = (T_1, T_2, T_1T_2, T_1T_2, T_1^2T_2, T_1T_2^2, T_1^2T_2^2)$ is a $\Gamma_{E(3; 3; 1, 1, 1)}$-contraction. 	\end{thm}
	
	\begin{proof} Define the map $\pi : \mathbb{C}^2 \to \mathbb{C}^7$ defined by
		\begin{equation}
			\begin{aligned}
				\pi(x, y) &= (x, y, xy, xy, x^2y, xy^2, x^2y^2).
			\end{aligned}
		\end{equation}
Let $$A=\begin{pmatrix}
x& 0 & 0\\0 & y & 0\\
0 & 0 & xy
\end{pmatrix}.$$  Suppose that $(x,y)\in \overline{\mathbb D}^2,$ then $\|A\|\leq 1.$ It follows from [Theorem $2.41$, \cite{apal1}] that $(x, y, xy, xy, x^2y, xy^2, x^2y^2)\in \Gamma_{E(3; 3; 1, 1, 1)}.$ Thus, we get $\pi(\overline{\mathbb{D}}^2) \subseteq \Gamma_{E(3; 3; 1, 1, 1)}$. For any $p \in \mathbb{C}[z_1, z_2,\dots,z_7],$ we observe that  $p \circ \pi$ is a rational function defined on $\overline{\mathbb D}^2.$
Observe that \begin{equation*}
			\begin{aligned}
				||p(\textbf{T})||
				&= ||p \circ \pi(T_1, T_2)||\\
				&\leqslant ||p \circ \pi||_{\infty, \overline{\mathbb{D}}^2} \,\, [{\rm{by~ von ~Neuman ~ineqality~for~}} \mathbb D^2] \\
				&= ||p||_{\infty, \pi(\overline{\mathbb{D}}^2)}\\
				&\leqslant ||p||_{\infty, \Gamma_{E(3; 3; 1, 1, 1)}}.
			\end{aligned}
		\end{equation*}
This shows that $\textbf{T}$ is a $\Gamma_{E(3; 3; 1, 1, 1)}$-contraction. This completes the proof.
	\end{proof}

\begin{thm}\label{contr}
Let $(T_1, T_2)$ be a pair of commuting contractions on a Hilbert space $\mathcal{H}$. Then the $7$-tuple of operators $\textbf{T} = (T_1, T_2, T_1T_2, T_1T_2, T_1^2T_2, T_1T_2^2, T_1^2T_2^2)$ always has $\Gamma_{E(3; 3; 1, 1, 1)}$-isometric dilation.

\end{thm}
\begin{proof}
Let $(V_1,V_2)$ be an Ando isometric dilation of $(T_1, T_2)$. Then it is easy to see that $(V_1, V_2, V_1V_2, V_1V_2, \\V_1^2V_2, V_1V_2^2, V_1^2V_2^2)$ is a $7$-tuple of commuting isometic lift of $\textbf{T} = (T_1, T_2, T_1T_2, T_1T_2, T_1^2T_2, T_1T_2^2, T_1^2T_2^2)$. It follows from [Theorem $4.4$, \cite{apal2}] that $(V_1, V_2, V_1V_2, V_1V_2, V_1^2V_2, V_1V_2^2, V_1^2V_2^2)$ is a
$\Gamma_{E(3; 3; 1, 1, 1)}$-isometry. This completes the proof.
\end{proof}

\begin{defn}\label{isometric dilation21}
A commuting $5$-tuple of operators $(W_1,W_2,W_3,\tilde{W}_1,\tilde{W}_2)$ acting on a Hilbert space $\mathcal K_1 \supseteq \mathcal H_1$ is said to be a $\Gamma_{E(3; 2; 1, 2)}$-isometric dilation of a $\Gamma_{E(3; 2;1, 2)}$-contraction $(S_1,S_2,S_3,\tilde{S}_1,\tilde{S}_2)$ acting on a Hilbert space $\mathcal H_1,$ if it satisfies the following properties:
\begin{itemize}
\item  $(W_1,W_2,W_3,\tilde{W}_1,\tilde{W}_2)$ is $\Gamma_{E(3; 2; 1, 2)}$-isometry;
\item $W_i^{*}|_{\mathcal H_1}=S_i^{*}$ for $1\leq i \leq 3$ and $\tilde{W}_j^{*}|_{\mathcal H_1}=\tilde{S}_j^{*}$ for $1\leq j \leq 2.$
\end{itemize}

\end{defn}	
It yields from the aforementioned definition that $(W_1,W_2,W_3,\tilde{W}_1,\tilde{W}_2)$ is $\Gamma_{E(3; 2; 1, 2)}$-isometric dilation of a $\Gamma_{E(3; 2;1, 2)}$-contraction $(S_1,S_2,S_3,\tilde{S}_1,\tilde{S}_2)$  is equivalent to saying that $(W_1^*,W_2^*,W_3^*,\tilde{W}_1^*,\tilde{W}_2^*)$ is a $\Gamma_{E(3; 2; 1, 2)} $-co-isometric extension of  $(S_1^*,S_2^*,S_3^*,\tilde{S}_1^*,\tilde{S}_2^*)$. Moreover, we call the dilation as minimal if $$\tilde{\mathcal K}_{0}=\overline{{\rm{span}}}\{ W_{3}^{n}h:h\in\mathcal H ~{\rm{and}}~n \in \mathbb N\cup\{0\}\}.$$
\begin{defn}\label{fundamental}
Let $(S_1, S_2, S_3, \tilde{S}_1, \tilde{S}_2)$ be a $5$-tuple of commuting bounded operators defined on some Hilbert space $\mathcal H_1$. The equations are as stated below:
\begin{equation}
			\begin{aligned}\label{funda1}
					&S_1 - \tilde{S}^*_2S_3 = D_{S_3}G_1D_{S_3},\,\, \tilde{S}_2 - S^*_1S_3 = D_{S_3}\tilde{G}_2D_{S_3},
				\end{aligned}
			\end{equation}
			$${\rm{and}}$$
		\begin{equation}
				\begin{aligned}\label{funda11}
				&\frac{S_2}{2} - \frac{\tilde{S}^*_1}{2}S_3 = D_{S_3}G_2D_{S_3}, \,\, \frac{\tilde{S}_1}{2} - \frac{S^*_2}{2}S_3 = D_{S_3}\tilde{G}_1D_{S_3},				\end{aligned}
			\end{equation}
where $G_1,G_2,\tilde{G}_1$ and $\tilde{G}_2$ in $\mathcal{B}(\mathcal{D}_{S_3}),$ are referred to as the  fundamental equations for $(S_1, S_2, S_3, \tilde{S}_1, \tilde{S}_2)$.		
\end{defn}
For any $z\in \mathbb C,$ we define the operators $\tilde{S}_z = S_1 + z\tilde{S}_2, \tilde{P}_z = zS_3$ and $\hat{S}_z = \frac{S_2}{2} + z\frac{\tilde{S}_1}{2}, \hat{P}_z = zS_3.$ 
	
	\begin{thm}[Theorem $2.6$, \cite{apal3}]\label{s1s2}
		Let $(S_1, S_2, S_3, \tilde{S}_1, \tilde{S}_2)$ be a $5$-tuple of commuting bounded operators defined on some Hilbert space $\mathcal H_1$. Then in the following $(1) \Rightarrow (2) \Rightarrow (3) \Rightarrow (4) \Rightarrow (5):$
		\begin{enumerate}
			\item $\mathbf{S} = (S_1, S_2, S_3, \tilde{S}_1, \tilde{S}_2)$ is a $\Gamma_{E(3; 2; 1, 2)}$-contraction.
			
			\item $(S_1, \tilde{S}_2, S_3)$ and $(\frac{S_2}{2}, \frac{\tilde{S}_1}{2}, S_3)$ are $\Gamma_{E(2; 2; 1, 1)}$-contractions.
	\item For every $z \in \mathbb{T}$, we have
			\begin{equation}
				\begin{aligned}
					&\rho_{G_{E(2; 2; 1, 1)}}(S_1, z\tilde{S}_2, zS_3) \geqslant 0 ~\text{and}~
					\rho_{G_{E(2; 2; 1, 1)}}(\tilde{S}_2, zS_1, zS_3) \geqslant 0,
				\end{aligned} \end{equation}
				\begin{equation}
				\begin{aligned}
					&\rho_{G_{E(2; 2; 1, 1)}}\left(\frac{S_2}{2}, z\frac{\tilde{S}_1}{2}, zS_3\right) \geqslant 0 ~\text{and}~
					\rho_{G_{E(2; 2; 1, 1)}}\left(\frac{\tilde{S}_1}{2}, z\frac{S_2}{2}, zS_3\right) \geqslant 0
				\end{aligned}
			\end{equation}
			and the spectral radius of $\tilde{S}_z$  and $\hat{S}_z$ are not bigger than $2$.
			
			\item The pair of operators $(\tilde{S}_z, \tilde{P}_z)$ and  $(\hat{S}_z, \hat{P}_z)$ are  $\Gamma_{E(2; 1; 2)}$-contractions for every $z \in \mathbb{T}$.
			
			\item The fundamental equations in \eqref{funda1} and \eqref{funda11} have unique solutions $G_1, \tilde{G}_2$ and $G_2,\tilde{G}_1$ in $\mathcal{B}(\mathcal{D}_{S_3})$, respectively. Moreover, the operators $G_1 + z\tilde{G}_2$ and $G_2 + z\tilde{G}_1$ have numerical radius not bigger than $1$ for every $z \in \mathbb{T}$.
		\end{enumerate}
	\end{thm}
	
The following theorem [Theorem $4.6$, \cite{apal3}] gives the sufficient conditions for the existence of  $\Gamma_{E(3; 2; 1, 2)}$-isometric dilation under the assumption that $\textbf{S} = (S_1, S_2, S_3, \tilde{S}_1, \tilde{S}_2)$ is a $\Gamma_{E(3; 2; 1, 2)}$-contraction, with its fundamental operators $G_1,2G_2,2\tilde{G}_1$ and $\tilde{G}_2$ which satisfy the following conditions:

		\begin{enumerate}
			\item[$(i)$] $[G_1,\tilde{G}_i]=0$ for $1 \leq i \leq 2,$ $[G_2,\tilde{G}_j]=0$ for $1 \leq j \leq 2,$ and $[G_1, G_2] = [\tilde{G}_1, \tilde{G}_2]  = 0$;
			
			\item[$(ii)$] $[G_1, G^*_1] = [\tilde{G}_2, \tilde{G}^*_2], [G_2, G^*_2] = [\tilde{G}_1, \tilde{G}^*_1], [G_1, \tilde{G}^*_1] = [G_2, \tilde{G}^*_2], [\tilde{G}_1, G^*_1] = [\tilde{G}_2, G^*_2],$\\$ [G_1, G^*_2] = [\tilde{G}_1, \tilde{G}^*_2], [G^*_1, G_2] = [\tilde{G}^*_1, \tilde{G}_2]$.
		\end{enumerate}		

\begin{thm}[Conditional Dilation of $\Gamma_{E(3; 2; 1, 2)}$-Contraction]\label{condilation1}
		Let $\textbf{S} = (S_1, S_2, S_3, \tilde{S}_1, \tilde{S}_2)$ be a $\Gamma_{E(3; 2; 1, 2)}$-contraction defined on a Hilbert space $\mathcal H$ with the fundamental operators $G_1,2G_2,2\tilde{G}_1$ and $\tilde{G}_2$ which satisfy the following conditions:

		\begin{enumerate}
			\item[$(i)$] $[G_1,\tilde{G}_i]=0$ for $1 \leq i \leq 2,$ $[G_2,\tilde{G}_j]=0$ for $1 \leq j \leq 2,$ and $[G_1, G_2] = [\tilde{G}_1, \tilde{G}_2]  = 0$;
			
			\item[$(ii)$] $[G_1, G^*_1] = [\tilde{G}_2, \tilde{G}^*_2], [G_2, G^*_2] = [\tilde{G}_1, \tilde{G}^*_1], [G_1, \tilde{G}^*_1] = [G_2, \tilde{G}^*_2], [\tilde{G}_1, G^*_1] = [\tilde{G}_2, G^*_2],$\\$ [G_1, G^*_2] = [\tilde{G}_1, \tilde{G}^*_2], [G^*_1, G_2] = [\tilde{G}^*_1, \tilde{G}_2]$.
		\end{enumerate}		
Let \begin{equation*}
			\begin{aligned}
				\mathcal{\tilde{K}} &= \mathcal{H} \oplus \mathcal{D}_{S_3} \oplus \mathcal{D}_{S_3} \oplus \dots = \mathcal{H} \oplus l^2(\mathcal{D}_{S_3}).
			\end{aligned}
		\end{equation*}
Suppose that $\textbf{W}=(W_1,W_2,W_3,\tilde{W}_1, \tilde{W}_2)$ is a $5$-tuple of bounded operators on $\tilde{K}$ by
		\begin{equation}\label{w3}
			\begin{aligned}
				&W_1 =
				\begin{bmatrix}
					S_1 & 0 & 0 & \dots\\
					\tilde{G}^*_2D_{S_3} & G_1 & 0 & \dots\\
					0 & \tilde{G}^*_2 & G_1 & \dots\\
					0 & 0 & \tilde{G}^*_2 & \dots\\
					\vdots & \vdots & \vdots & \ddots
				\end{bmatrix}, \,\,
				W_2 =
				\begin{bmatrix}
					S_2 & 0 & 0 & \dots\\
					2\tilde{G}^*_1D_{S_3} & 2G_2 & 0 & \dots\\
					0 & 2\tilde{G}^*_1 & 2G_2 & \dots\\
					0 & 0 & 2\tilde{G}^*_1 & \dots\\
					\vdots & \vdots & \vdots & \ddots
				\end{bmatrix},
				W_3 =
				\begin{bmatrix}
					S_3 & 0 & 0 & \dots\\
					D_{S_3} & 0 & 0 & \dots\\
					0 & I & 0 & \dots\\
					0 & 0 & I & \dots\\
					\vdots & \vdots & \vdots & \ddots
				\end{bmatrix},\\
				&\hspace{2cm}
				\tilde{W}_1 =
				\begin{bmatrix}
					\tilde{S}_1 & 0 & 0 & \dots\\
					2G^*_2D_{S_3} & 2\tilde{G}_1 & 0 & \dots\\
					0 & 2G^*_2 &2 \tilde{G}_1 & \dots\\
					0 & 0 & 2G^*_2 & \dots\\
					\vdots & \vdots & \vdots & \ddots
				\end{bmatrix} \,\,{\rm{and}}~
				\tilde{W}_2 =
				\begin{bmatrix}
					\tilde{S}_2 & 0 & 0 & \dots\\
					G^*_1D_{S_3} & \tilde{G}_2 & 0 & \dots\\
					0 & G^*_1 & \tilde{G}_2 & \dots\\
					0 & 0 & G^*_1 & \dots\\
					\vdots & \vdots & \vdots & \ddots
				\end{bmatrix}.
			\end{aligned}
		\end{equation}
Then we have the following: 
\begin{enumerate}
\item $\textbf{W}$ is a minimal $\Gamma_{E(3; 2; 1, 2)}$-isometric dilation of $\textbf{S}$.
		
\item If there exists a $\Gamma_{E(3; 2; 1, 2)}$-isometric dilation $\textbf{X} = (X_1, X_2, X_3, \tilde{X}_1, \tilde{X}_2)$ of $\textbf{S}$ such that $X_3$ is a minimal isometric dilation of $S_3$, then $\textbf{X}$ is unitarily equivalent to $\textbf{W}$. Moreover, the above identities  $(i)$ and $(ii)$ are also valid.
\end{enumerate}
\end{thm}
We will demonstrate the necessary conditions for the existence of $\Gamma_{E(3; 2; 1, 2)}$-isometric dilation. 

\begin{thm}\label{Necessary Conditions 2}
		Let $\textbf{S} = (S_1, S_2, S_3, \tilde{S}_1, \tilde{S}_2) $ be a $\Gamma_{E(3; 2; 1, 2)} $-contraction on a Hilbert space $\mathcal{H}$ and $G_1, 2G_2, 2\tilde{G}_1, \tilde{G}_2$ be the fundamental operators of $\textbf{S}$. Then each of the following conditions is necessary for $\textbf{S}$ to have a $\Gamma_{E(3; 2; 1, 2)} $-isometric dilation:
		\begin{enumerate}
			\item The tuple $(G_1, 2G_2, 2\tilde{G}_1, \tilde{G}_2)$ has a joint dilation to a  $4$-tuple of commuting subnormal operators \\$(H_1, 2H_2, 2\tilde{H}_1, \tilde{H}_2)$, that is, there exists an isometric embedding $\Lambda_0$ of $\mathcal{D}_{S_3}$ into a larger Hilbert space $\mathcal{F}$ so that $G_1 = \Lambda^*_0H_1\Lambda_0, G_2 = \Lambda^*_0H_2\Lambda_0, \tilde{G}_1 = \Lambda^*_0\tilde{H}_1\Lambda_0, \tilde{G}_2 = \Lambda^*_0\tilde{H}_2\Lambda_0$, where $(H_1, 2H_2, 2\tilde{H}_1, \tilde{H}_2)$ can be extended to  $4$-tuple of commuting normal operators $(M_1, 2M_2, 2\tilde{M}_1, \tilde{M}_2)$ with Taylor joint spectrum is contained in $\{(z_1, 2z_2, 2\tilde{z}_1, \tilde{z}_2) : |z_1| = |\tilde{z}_2| \leqslant 1, |z_2| = |\tilde{z}_1| \leqslant 1\}$.
			
			\item $(\tilde{G}^*_2D_{S_3}\tilde{S}_2 - G^*_1D_{S_3}S_1)|_{Ker D_{S_3}} = 0$.
			
			\item[$(2^{'})$] $(\tilde{G}^*_2G^*_1 - G^*_1\tilde{G}^*_2)D_{S_3}S_3|_{Ker D_{S_3}} = 0$.
			
			\item $(G^*_2D_{S_3}S_2 - \tilde{G}^*_1D_{S_3}\tilde{S}_1)|_{Ker D_{S_3}} = 0$.
			
			\item[$(3^{'})$] $(G^*_2\tilde{G}^*_1 - \tilde{G}^*_1G^*_2)D_{S_3}S_3|_{Ker D_{S_3}} = 0$.
			
			\item $(\tilde{G}^*_2D_{S_3}S_2 - 2\tilde{G}^*_1D_{S_3}S_1)|_{Ker D_{S_3}} = 0$.
			
			\item[$(4^{'})$] $(\tilde{G}^*_2\tilde{G}^*_1 - \tilde{G}^*_1\tilde{G}^*_2)D_{S_3}S_3|_{Ker D_{S_3}} = 0$.
			
			\item $(2G^*_2D_{S_3}\tilde{S}_2 - G^*_1D_{S_3}\tilde{S}_1)|_{Ker D_{S_3}} = 0$.
			
			\item[$(5^{'})$] $(G^*_2G^*_1 - G^*_1G^*_2)D_{S_3}S_3|_{Ker D_{S_3}} = 0$.
			
			\item $(\tilde{G}^*_2D_{S_3}\tilde{S}_1 - 2G^*_2D_{S_3}S_1)|_{Ker D_{S_3}} = 0$.
			
			\item[$(6^{'})$] $(\tilde{G}^*_2G^*_2 - G^*_2\tilde{G}^*_2)D_{S_3}S_3|_{Ker D_{S_3}} = 0$.
			
			\item $(2\tilde{G}^*_1D_{S_3}\tilde{S}_2 - G^*_1D_{S_3}S_2)|_{Ker D_{S_3}} = 0$.
			
			\item[$(7^{'})$] $(\tilde{G}^*_1G^*_1 - G^*_1\tilde{G}^*_1)D_{S_3}S_3|_{Ker D_{S_3}} = 0$.
		\end{enumerate}
	\end{thm}
	
	\begin{proof}
		Let  $\textbf{W} = (W_1, W_2, W_3, \tilde{W}_1, \tilde{W}_2)$ is a $\Gamma_{E(3; 2; 1, 2)}$-isometric dilation of $\textbf{S}.$ It is noteworthy that $\Gamma_{E(3; 2; 1, 2)}$ is a polynomially convex [Theorem $4.1$, \cite{Zapalowski}]. Therefore, it suffices to consider polynomials instead of the entire algebra $\mathcal O(\Gamma_{E(3; 2; 1, 2)}),$ and we can assume, without loss of generality, that \begin{equation*}
			\begin{aligned}
				\mathcal{K} &= \overline{\rm{span}}\{W^{n_1}_1W^{n_2}_2W^{n_3}_3\tilde{W}^{m_1}_1 \tilde{W}^{m_2}_2h : h \in \mathcal{H}, n_1, n_2, n_3, m_1, m_2 \in \mathbb{N} \cup \{0\}\}.
			\end{aligned}
		\end{equation*}
By definition, we have
		\begin{equation*}
			\begin{aligned}
				(W^*_1, W^*_2, W^*_3, \tilde{W}^*_1, \tilde{W}^*_2)|_{\mathcal{H}} &= (S^*_1, S^*_2, S^*_3, \tilde{S}^*_1, \tilde{S}^*_2).
			\end{aligned}
		\end{equation*}
Let the $2 \times 2$ block operator matrix of $W_i$'s for $1 \leqslant i \leqslant 3$ and $\tilde{W}_j$'s for $1 \leqslant j \leqslant 2$ be of the form
		\begin{equation}\label{M 1}
			\begin{aligned}
				W_1 &=
				\begin{pmatrix}
					S_1 & 0\\
					E_1 & H_1
				\end{pmatrix},
				W_2 =
				\begin{pmatrix}
					S_2 & 0\\
					E_2 & 2H_2
				\end{pmatrix},
				W_3 =
				\begin{pmatrix}
					S_3 & 0\\
					E_3 & H_3
				\end{pmatrix},\\
				&\hspace{0.5cm}
				\tilde{W}_1 =
				\begin{pmatrix}
					\tilde{S}_1 & 0\\
					\tilde{E}_1 & 2\tilde{H}_1
				\end{pmatrix} \,\, \text{and} \,\,
				\tilde{W}_2 =
				\begin{pmatrix}
					\tilde{S}_2 & 0\\
					\tilde{E}_2 & \tilde{H}_2
				\end{pmatrix}
			\end{aligned}
		\end{equation}
with respect to the decomposition $\mathcal{K} = \mathcal{H} \oplus (\mathcal{K} \ominus \mathcal{H})$ of $\mathcal{K}$. Since $W_3$ is an isometry, it follows from \eqref{M 1} that
		\begin{equation}\label{M 2}
			\begin{aligned}
				S^*_3S_3 + E^*_3E_3 &= I_{\mathcal{H}}, H^*_3H_3 = I_{\mathcal{K} \ominus \mathcal{H}}.
			\end{aligned}
		\end{equation}
It yields from \eqref{M 2} that there exists an isometry $\Lambda_0 : \mathcal{D}_{S_3} \to \mathcal{K} \ominus \mathcal{H}$ such that
		\begin{equation}\label{M 3}
			\begin{aligned}
				\Lambda_0 D_{S_3} = E_3.
			\end{aligned}
		\end{equation}
Since $\textbf{W} = (W_1, W_2, W_3, \tilde{W}_1, \tilde{W}_2)$ is a $\Gamma_{E(3; 2; 1, 2)}$-isometry, it implies from [Theorem $4.5$, \cite{apal2}] that
\begin{equation}\label{M111}
 W_1 = \tilde{W}^*_2W_3, \tilde{W}_2=W_1^*W_3,  W_2 = \tilde{W}^*_1W_3~{\rm{and}}~\tilde{W}_1=W_2^*W_3.
 \end{equation}
 We deduce from \eqref{M 1} and \eqref{M111} that 
 \begin{equation}\label{M 4}
			\begin{aligned}
				S_1 = \tilde{S}^*_2S_3 + \tilde{E}^*_2E_3, E_1 = \tilde{H}^*_2E_3, \tilde{E}^*_2H_3 = 0, H_1 = \tilde{H}^*_2H_3,
				\end{aligned}
				\end{equation}
			 \begin{equation}\label{M 41}
			\begin{aligned}
			\tilde{S}_2 = S^*_1S_3 + E^*_1E_3, \tilde{E}_2 = H^*_1E_3, E^*_1H_3 = 0, \tilde{H}_2 = H^*_1H_3,
			\end{aligned}
			\end{equation}
 \begin{equation}\label{M 42}
			\begin{aligned}
					S_2 = \tilde{S}^*_1S_3 + \tilde{E}^*_1E_3, E_2 = 2\tilde{H}^*_1E_3, \tilde{E}^*_1H_3 = 0, H_2 = \tilde{H}^*_1H_3
					\end{aligned}
			\end{equation}	$${\rm{and}}$$				
			 \begin{equation}\label{M 43}
			\begin{aligned}
			\tilde{S}_1 = S^*_2S_3 + E^*_2E_3, \tilde{E}_1 = 2H^*_2E_3, E^*_2H_3 = 0, \tilde{H}_1 = H^*_2H_3.
			\end{aligned}
		\end{equation}
It follows from \eqref{funda1},\eqref{funda11},\eqref{M 3},\eqref{M 4},\eqref{M 41},\eqref{M 42}, \eqref{M 43} and Theorem \ref{s1s2} that	\begin{equation}\label{M 6}
			\begin{aligned}
				D_{S_3}G_1D_{S_3} = S_1 - \tilde{S}^*_2S_3 = \tilde{E}^*_2E_3 = E^*_3H_1E_3 = D_{S_3}\Lambda^*_0H_1\Lambda_0 D_{S_3},
				\end{aligned}
		\end{equation}
\begin{equation}\label{M 61}
			\begin{aligned}
					D_{S_3}\tilde{G}_2D_{S_3} = \tilde{S}_2 - S^*_1S_3 = E^*_1E_3 = E^*_3\tilde{H}_2E_3 = D_{S_3}\Lambda^*_0\tilde{H}_2\Lambda_0 D_{S_3},
			\end{aligned}
		\end{equation}
		\begin{equation}\label{M 7}
			\begin{aligned}
				2D_{S_3}G_2D_{S_3} = S_2 - \tilde{S}^*_1S_3 = \tilde{E}^*_1E_3 = 2E^*_3H_2E_3 = 2D_{S_3}\Lambda^*_0H_2\Lambda_0 D_{S_3},
				\end{aligned}
		\end{equation}
				$${\rm{and}}$$
				\begin{equation}\label{M 71}
			\begin{aligned}
				2D_{S_3}\tilde{G}_1D_{S_3} = \tilde{S}_1 - S^*_2S_3 = E^*_2E_3 = 2E^*_3\tilde{H}_1E_3 = 2D_{S_3}\Lambda^*_0\tilde{H}_1\Lambda_0 D_{S_3}.
				\end{aligned}
		\end{equation}
		By uniqueness of the fundamental operators  $G_1, 2G_2, 2\tilde{G}_1, \tilde{G}_2$, we conclude from \eqref{M 6}, \eqref{M 61}, \eqref{M 7} and \eqref{M 71} that 
		\begin{equation}\label{M 8}
			\begin{aligned}
				G_1 &= \Lambda^*_0H_1\Lambda_0, G_2 = \Lambda^*_0H_2\Lambda_0, \tilde{G}_1 = \Lambda^*_0\tilde{H}_1\Lambda_0, \tilde{G}_2 = \Lambda^*_0\tilde{H}_2\Lambda_0.
			\end{aligned}
		\end{equation}
From \eqref{M 2}, \eqref{M 4} and \eqref{M 42}, it is evident
\begin{equation}\label{M 9}
			\begin{aligned}
				H^*_3H_3 = I_{\mathcal{K} \ominus \mathcal{H}}, H_1 = \tilde{H}^*_2H_3, \,\, \text{and} \,\, H_2 = \tilde{H}^*_1H_3.
			\end{aligned}
			\end{equation}
As $\textbf{W}=(W_1, W_2, W_3, \tilde{W}_1, \tilde{W}_2)$ is a $\Gamma_{E(3; 2; 1, 2)}$-isometry and $(W_1, W_2, W_3, \tilde{W}_1, \tilde{W}_2)|_{\mathcal{K} \ominus \mathcal{H}} = (H_1, 2H_2, H_3, 2\tilde{H}_1, \tilde{H}_2)$, it indicates that $(H_1, 2H_2, H_3, 2\tilde{H}_1, \tilde{H}_2)$ is a $\Gamma_{E(3; 2; 1, 2)} $-contraction. As $\textbf{H} = (H_1, 2H_2, H_3, 2\tilde{H}_1, \tilde{H}_2)$ is a $\Gamma_{E(3; 2; 1, 2)}$-contraction, we conclude from \eqref{M 9} that $\textbf{H} = (H_1, 2H_2, H_3, 2\tilde{H}_1, \tilde{H}_2)$ is also a $\Gamma_{E(3; 2; 1, 2)} $-isometry, and so by definition of $\Gamma_{E(3; 2; 1, 2)}$-isometry, $\textbf{H}$ has a $\Gamma_{E(3; 2; 1, 2)}$-unitary extension $\textbf{M} = (M_1, 2M_2, M_3, 2\tilde{M}_1, \tilde{M}_2)$ on a larger Hilbert space. Since $\textbf{M} = (M_1, 2M_2, M_3, 2\tilde{M}_1, \tilde{M}_2)$ is a $\Gamma_{E(3; 2; 1, 2)}$-unitary, it follows from the definition of $\Gamma_{E(3; 2; 1, 2)}$-unitary that the Taylor joint spectrum $\sigma(\textbf{M})$ of $\textbf{M}$ is contained in $K_1$ and $M_1, M_2, M_3, \tilde{M}_1, \tilde{M}_2$ are commuting normal operators. By ignoring the third co-ordinate, we conclude that $\sigma(M_1, 2M_2, 2\tilde{M}_1, \tilde{M}_2)$ is contained in $\{(z_1, 2z_2, 2\tilde{z}_1, \tilde{z}_2) : |z_1| = |\tilde{z}_2| \leqslant 1, |z_2| = |\tilde{z}_1| \leqslant 1\}, $ and part $(1)$ follows.
		
We demonstrate only conditions $(2)$ and $(2^{\prime})$, as the conditions $(3),(3^{\prime}),(4),(4^{\prime}),(5),(5^{\prime}),(6),(6^{\prime}),(7)$ and $(7^{\prime})$ are satisfied in a similar manner.  As $W_1\tilde{W}_2=\tilde{W}_2W_1$, it follows from \eqref{M 1} that 
		\begin{equation}\label{M 10}
			\begin{aligned}
				E_1\tilde{S}_2 +H_1\tilde{E}_2 &=
				\tilde{H}_2E_1 +\tilde{E}_2S_1.
			\end{aligned}
		\end{equation}
It implies from \eqref{M 4}, \eqref{M 41} and \eqref{M 10} that 		
\begin{equation}\label{M 111}
			\begin{aligned}
				E_1\tilde{S}_2 +H_1\tilde{E}_2 &=
				\tilde{H}_2E_1 +\tilde{E}_2S_1.
			\end{aligned}
		\end{equation}		
From \eqref{M 4},\eqref{M 41} and \eqref{M 3},  we see that 
\begin{equation}\label{M 121}
			\begin{aligned}
				E_1\tilde{S}_2 -\tilde{E}_2S_1 &=\tilde{H}_2^*E_3\tilde{S}_2-H_1^*E_3S_1\\&=\tilde{H}_2^*\Lambda_0 D_{S_3}\tilde{S}_2-H_1^*\Lambda_0 D_{S_3}S_1.
				\end{aligned}
		\end{equation}		
Also, it yields from \eqref{M 4}, \eqref{M 41} and \eqref{M 3} that	
\begin{equation}\label{M 1211}
			\begin{aligned}
				\tilde{H}_2E_1 -H_1\tilde{E}_2 &=\tilde{H}_2\tilde{H}_2^*E_3-H_1H_1^*E_3\\&=(\tilde{H}_2\tilde{H}_2^*-H_1H_1^*)\Lambda_0 D_{S_3}S_1.
				\end{aligned}
		\end{equation}		
From \eqref{M 111}, \eqref{M 121} and \eqref{M 1211}, we have 		
\begin{equation}\label{M 211}
			\begin{aligned}
				\tilde{H}^*_2\Lambda_0 D_{S_3}\tilde{S}_2 - H^*_1\Lambda_0 D_{S_3}S_1 &= (\tilde{H}_2\tilde{H}^*_2 - H_1H^*_1)\Lambda_0 D_{S_3}.
			\end{aligned}
			\end{equation}
By multiplying left side  of \eqref{M 211}  by $\Lambda^*_0$ and by using \eqref{M 8},  we deduce that
		\begin{equation}\label{M 311}
			\begin{aligned}
			\tilde{G}^*_2 D_{S_3}\tilde{S}_2 - G^*_1 D_{S_3}S_1&=	\Lambda^*_0\tilde{H}^*_2\Lambda_0 D_{S_3}\tilde{S}_2 - \Lambda^*_0H^*_1\Lambda_0 D_{S_3}S_1 \\&= \Lambda^*_0(\tilde{H}_2\tilde{H}^*_2 - H_1H^*_1)\Lambda_0 D_{S_3}.
			\end{aligned}
		\end{equation}
	Therefore, from \eqref{M 311}, we  conclude that
		\begin{equation*}
			\begin{aligned}
				(\tilde{G}^*_2 D_{S_3}\tilde{S}_2 - G^*_1 D_{S_3}S_1)|_{Ker D_{S_3}} &= 0,
			\end{aligned}
		\end{equation*}
		part $(2)$ follows.

 Observe that
		\begin{equation}\label{M 13}
			\begin{aligned}
				\Lambda^*_0(\tilde{H}_2\tilde{H}^*_2 - H_1H^*_1)\Lambda_0 D_{S_3} &=
				\tilde{G}^*_2 D_{S_3}\tilde{S}_2 - G^*_1 D_{S_3}S_1\\
				&= \tilde{G}^*_2(\tilde{G}_2D_{S_3} + G^*_1D_{S_3}S_3) - G^*_1(G_1D_{S_3} + \tilde{G}^*_2D_{S_3}S_3)\\
				&= (\tilde{G}^*_2\tilde{G}_2 - G^*_1G_1)D_{S_3} + (\tilde{G}^*_2G^*_1 - G^*_1\tilde{G}^*_2)D_{S_3}S_3.
			\end{aligned}
		\end{equation}
It follows from \eqref{M 13} that
		\[(\tilde{G}^*_2G^*_1 - G^*_1\tilde{G}^*_2)D_{S_3}S_3|_{Ker D_{S_3}} = 0,\]
		part $(2^{\prime})$ follows.
		
From above observations, we also conclude that $(2) \Leftrightarrow (2^{'}).$ Similarly, we can show that $(3) \Leftrightarrow (3^{'}), (4) \Leftrightarrow (4^{'}), (5) \Leftrightarrow (5^{'}), (6) \Leftrightarrow (6^{'}), (7) \Leftrightarrow (7^{'}).$ This completes the proof.
	\end{proof}
We discuss a class of $\Gamma_{E(3; 2; 1, 2)}$-contractions that are always dilate, specifically those of the form $\textbf{S} = (S_1,  S_1S_2+S_1^2S_2,  S_1^2S_2^2, S_2+S_1S_2,S_1S_2^2),$ where $(S_1,S_2)$ is  a pair of contractions.
\begin{thm}\label{Thm Exam 2}
		Let $(S_1, S_2)$ be a pair of commuting contractions on a Hilbert space $\mathcal{H}$. Then the $5$-tuple of operators $\textbf{S} = (S_1,  S_1S_2+S_1^2S_2,  S_1^2S_2^2, S_2+S_1S_2,S_1S_2^2)$ is a $\Gamma_{E(3; 2; 1, 2)}$-contraction.	\end{thm}
	
	\begin{proof}
Observe that a point $(x_1, \dots, x_7) \in \Gamma_{E(3; 3; 1, 1, 1)}$ if and only if $(x_1, x_3 + \eta x_5, \eta x_7, x_2 + \eta x_4, \eta x_6) \in \Gamma_{E(3; 2; 1, 2)}$ for all $\eta \in \overline{\mathbb{D}}$ [Theorem $2.48$, \cite{apal1}]. For $\eta \in \overline{\mathbb{D}}$, we define the map $\pi_{\eta} : \mathbb{C}^7 \to \mathbb{C}^5$ by
		\begin{equation*}
			\begin{aligned}
				\pi_{\eta}(x_1, \dots, x_7)
				&= (x_1, x_3 + \eta x_5, \eta x_7, x_2 + \eta x_4, \eta x_6).
			\end{aligned}
		\end{equation*}
It is important to note from Theorerm \ref{Thm Example 1} that  $\pi(\overline{\mathbb{D}}^2) \subseteq \Gamma_{E(3; 3; 1, 1,1)}$.  Hence we have  $\pi_{\eta}\circ\pi(\overline{\mathbb{D}}^2) \subseteq \Gamma_{E(3; 2; 1, 2)}$. In particular, for $\eta =1$, we have $\pi_{1}\circ\pi(\overline{\mathbb{D}}^2) \subseteq \Gamma_{E(3; 2; 1, 2)}$. Let $p$ be any polynomial in  $\mathbb{C}[z_1, z_2,\dots,z_5].$ Then $p \circ \pi_1 \circ \pi$ is a polynomial on $\overline{\mathbb D^2}$ and we deduce that
		\begin{equation*}
			\begin{aligned}
				||p(\textbf{S})|| &= ||p \circ \pi_1\circ \pi(S_1, S_2)||\\
				&\leqslant ||p \circ \pi_1 \circ \pi||_{\infty, \overline{\mathbb{D}}^2}\\
				&= ||p||_{\infty, \pi_1\circ \pi(\overline{\mathbb{D}}^2)}\\
				&\leqslant ||p||_{\infty, \Gamma_{E(3; 2; 1, 2)}}.
			\end{aligned}
		\end{equation*}
This shows that $\textbf{S}$ is a $\Gamma_{E(3; 2; 1, 2)}$-contraction. This completes the proof.
\end{proof}
\begin{thm}
Let $(S_1, S_2)$ be a pair of commuting contractions on a Hilbert space $\mathcal{H}$. Then the $5$-tuple of operators $\textbf{S} = (S_1,  S_1S_2+S_1^2S_2,  S_1^2S_2^2, S_2+S_1S_2,S_1S_2^2)$ always has $\Gamma_{E(3; 2; 1, 2)}$-dilation.
\end{thm}
\begin{proof}
Let $(V_1,V_2)$ be an Ando isometric dilation of $(S_1, S_2)$. Then it is easy to see that  $(V_1,  V_1V_2+V_1^2V_2,  V_1^2V_2^2, V_2+V_1V_2,V_1V_2^2)$ is a $5$-tuple of commuting isometic lift of  $\textbf{S} = (S_1,  S_1S_2+S_1^2S_2,  S_1^2S_2^2, S_2+S_1S_2,S_1S_2^2)$. It follows from [Theorem $4.5$, \cite{apal2}] that $(V_1,  V_1V_2+V_1^2V_2,  V_1^2V_2^2, V_2+V_1V_2,V_1V_2^2)$ is a  $\Gamma_{E(3; 2; 1, 2)}$-isometry. This completes the proof.
\end{proof}

	\section{$\bar{\mathcal{P}}$-Contraction and Their Isometric Dilation: Necessary Conditions}\label{Sec 3}
	
	In this section, we establish a necessary condition for the existence of a $\mathcal{\bar{P}}$-isometric dilation. The following theorem from \cite{Roy} guarantees the existence and uniqueness of the fundamental operator for a $\Gamma_{E(2;1;2)}$-contraction.
\begin{thm}[Theorem 4.2, \cite{Roy}]\label{FO of Gamma}
		Let $(T_1, T_2)$ be a $\Gamma_{E(2;1;2)}$-contraction. Then there exists a unique solution $X$ to the fundamental
		equation $T_1 - T^*_1T_2 = D_{T_2}XD_{T_2}.$ Furthermore, the numerical radius of $X$ is less than or equal to one.
	\end{thm}
%	\begin{prop}\label{Pentablock 1}
%		Let $\textbf{P} = (P_1, P_2, P_3)$ and $\tilde{\textbf{P}} = (\tilde{P}_1, \tilde{P}_2, \tilde{P}_3)$ be two$\mathcal{\bar{P}}$-contractions and $X, Y$ are the fundamental operators of $(P_2, P_3)$ and $(\tilde{P}_2, \tilde{P}_3)$ respectively. If $\textbf{P}$ and $\tilde{\textbf{P}}$ are unitarily equivalent then $P_1, \tilde{P}_1$ are unitarily equivalent, and $X, Y$ are unitarily equivalent.
%	\end{prop}
%	
%	\begin{proof}
%		Suppose $\textbf{P} = (P_1, P_2, P_3)$ and $\tilde{\textbf{P}} = (\tilde{P}_1, \tilde{P}_2, \tilde{P}_3)$ are unitarily equivalent$\mathcal{\bar{P}}$-contractions. Then there exists a unitary $U$ such that $UP_i = \tilde{P}_iU$ for $1 \leqslant i \leqslant 3$. By the characterization of$\mathcal{\bar{P}}$-contractions we have $(P_2, P_3)$ and $(\tilde{P}_2, \tilde{P}_3)$ are $\Gamma$-contractions. And so $(P_2, P_3)$ and $(\tilde{P}_2, \tilde{P}_3)$ are unitarily equivalent. Thus by Proposition 4.2 of \cite{T. Bhattacharyya} we conclude that the fundamental operators $X, Y$ of $(P_2, P_3)$ and $(\tilde{P}_2, \tilde{P}_3)$ respectively are also unitarily equivalent. This completes the proof.
%	\end{proof}
We now define the $\mathcal{\bar{P}}$-isometric dilation of a $\mathcal{\bar{P}}$-contraction $(P_1, P_2, P_3)$.
	\begin{defn}\label{Defn 3}
A commuting triple of bounded operators $(R_1, R_2, R_3)$ on a Hilbert space $\mathcal{K}$ containing $\mathcal{H}$ is called a $\mathcal{\bar{P}}$-isometric dilation of a $\mathcal{\bar{P}}$-contraction $(P_1, P_2, P_3)$ on the Hilbert space $\mathcal{H}$ if		\begin{itemize}
			\item $(R_1, R_2, R_3)$ is a $\mathcal{\bar{P}}$-isometry;
			
			\item $R^*_i|_{\mathcal{H}} = P^*_i$ for $1 \leqslant i \leqslant 3$.
		\end{itemize}
	\end{defn}
It yields from the aforementioned definition that $(R_1, R_2, R_3)$ is a $\mathcal{\bar{P}}$-isometric dilation of a $\mathcal{\bar{P}}$-contraction $(P_1, P_2, P_3)$ is equivalent to saying that $(R^*_1, R^*_2, R^*_3)$ is a $\mathcal{\bar{P}}$-co-isometric extension of $(P^*_1, P^*_2, P^*_3)$. Furthermore, if
	\[\mathcal{K} = \overline{\Span}\{R^{n_1}_1R^{n_2}_2R^{n_3}_3h : h \in \mathcal{H}, n_1, n_2, n_3 \in \mathbb{N} \cup \{0\}\}\]
	then we call it the \textit{minimal $\mathcal{\bar{P}}$-isometric dilation}.
We will now demonstrate the necessary condition for a $\mathcal{\bar{P}}$-isometric dilation.
	\begin{thm}\label{Necessary Conditions 3}
		Let $(P_1, P_2, P_3)$ be a $\mathcal{\bar{P}}$-contraction on a Hilbert space $\mathcal{H}$ and $X \in \mathcal{B}(\mathcal{D}_{P_3})$ be the fundamental operator of  $(P_1, P_2, P_3)$. Then each of the following conditions are necessary for $(P_1, P_2, P_3)$ to have a $\mathcal{\bar{P}}$-isometric dilation of $(P_1, P_2, P_3)$:
		\begin{enumerate}
			\item The fundamental operator $X$ has a Halmos dilation to a subnormal operator $N_2$, that is, there exists an isometric embedding $\Theta$ from $\mathcal{D}_{P_3}$ to a larger Hilbert space $\mathcal{F}$ so that $X = \Theta^*N_2\Theta$, and there exist subnormal operators $N_1, N_3$ on $\mathcal{F}$ such that $N_1, N_2, N_3$ commute and $(N_1, N_2, N_3)$ can be extended to a commuting triple of normal operators $(U_1, U_2, U_3)$ with the Taylor joint spectrum of $(U_1, U_2, U_3)$ contained in $b\mathcal{P}$, the distinguished boundary of $\mathcal{\bar{P}}$.
			
			\item $(X D_{P_3}P_3 -  D_{P_3}P_2)|_{ker D_{P_3}} = 0$.
		\end{enumerate}
	\end{thm}
	
	\begin{proof}
		Suppose that $(R_1, R_2, R_3)$ is a $\mathcal{\bar{P}}$-isometric dilation of the $\mathcal{\bar{P}}$-contraction $(P_1, P_2, P_3).$  It is important to note that $\mathcal P$ is a polynomially convex [Theorem $6.3$, \cite{Young}]. Therefore, it suffices to work with polynomials  rather than the entire algebra $\mathcal O(\mathcal P)$, and we can assume, without loss of generality, that,
		\[\mathcal{K} = \overline{\rm{span}}\{R^{n_1}_1R^{n_2}_2R^{n_3}_3h : h \in \mathcal{H}, n_1, n_2, n_3 \in \mathbb{N} \cup \{0\}\}.\] 
According to the definition, we have
		\[(R^*_1, R^*_2, R^*_3)|_{\mathcal{H}} = (P^*_1, P^*_2, P^*_3).\]
Let the $2 \times 2$ block operator matrix of $R_i$ be of the form 
		\begin{equation}\label{U 1}
			\begin{aligned}
				R_i &=
				\begin{pmatrix}
					P_i & 0\\
					B_i & N_i
				\end{pmatrix} \,\, \text{for} \,\, 1 \leqslant i \leqslant 3,
			\end{aligned}
		\end{equation}
		with respect to the decomposition $\mathcal{K} = \mathcal{H} \oplus (\mathcal{K} \ominus \mathcal{H})$. Since $(R_1, R_2, R_3)$ is a $\mathcal{\bar{P}}$-isometry, it follows from [Theorem $5.2$, \cite{Jindal}]  that $(R_2, R_3)$ is a $\Gamma_{E(2;1;2)}$-isometry and $R^*_1R_1 = I - \frac{1}{4}R^*_2R_2$. As $R_3$ is an isometry,  it implies from \eqref{U 1} that 	\begin{equation}\label{U 2}
			\begin{aligned}
				P^*_3P_3 + B^*_3B_3 &= I_{\mathcal{H}}, \,\, N^*_3N_3 = I_{\mathcal{K} \ominus \mathcal{H}}, \,\, \text{and} \,\, N^*_3B_3 =  0.
			\end{aligned}
		\end{equation}
It yields from $(\ref{U 2})$  that there exists an isometry $\Theta : \mathcal{D}_{P_3} \to \mathcal{K} \ominus \mathcal{H}$ such that 
\begin{equation}\label{U 11}\Theta D_{P_3} = B_3.\end{equation}
As $(R_2, R_3)$ is a $\Gamma_{E(2;1;2)}$-isometry, from  [Theorem $2.6$, \cite{ay1}], we get $R_2 = R^*_2R_3$. 
Observe that
		\begin{equation}\label{U 3}
			\begin{aligned}
				R_2=&\begin{pmatrix}
					P_2 & 0\\
					B_2 & N_2
				\end{pmatrix} \\&=
				\begin{pmatrix}
					P^*_2 & B^*_2\\
					0 & N^*_2
				\end{pmatrix}
				\begin{pmatrix}
					P_3 & 0\\
					B_3 & N_3
				\end{pmatrix}\\&=\begin{pmatrix}
					P^*_2P_3 + B^*_2B_3 & B^*_2N_3\\
					N^*_2B_3 & N^*_2N_3
				\end{pmatrix}.
			\end{aligned}
		\end{equation}
It implies from \eqref{U 3} that
		\begin{equation}\label{U 4}
			\begin{aligned}
				P_2 &= P^*_2P_3 + B^*_2B_3, \,\, B^*_2N_3 = 0, \,\, N^*_2B_3 = B_2 \,\, \text{and} \,\, N_2 = N^*_2N_3.
			\end{aligned}
		\end{equation}
As $(P_2, P_3)$ is a $\Gamma_{E(2;1;2)}$-contraction,  it follows from Theorem \ref{FO of Gamma} and \eqref{U 4} that 
		\begin{equation}\label{U 5}
			\begin{aligned}
				D_{P_3}XD_{P_3} &= P_2 - P^*_2P_3 = B^*_2B_3 = B^*_3N_2B_3 = D_{P_3}\Theta^*N_2\Theta D_{P_3}.
			\end{aligned}
		\end{equation}
By the uniqueness of the fundamental operator $X,$ we have
		\begin{equation}\label{U 6}
			\begin{aligned}
				X &= \Theta^*N_2\Theta.
			\end{aligned}
		\end{equation}
Since $(R_1, R_2, R_3)$ is a $\mathcal{\bar{P}}$-isometry and $(R_2, R_3)_{|_{\mathcal{K} \ominus \mathcal{H}}}=(N_2,N_3),$  it implies that $(N_2, N_3)$ is a $\Gamma_{E(2;1;2)}$-contraction. Since $(N_2, N_3)$ is a $\Gamma_{E(2;1;2)}$-contraction and $N^*_3N_3 = I_{\mathcal{K} \ominus \mathcal{H}},$ it implies from [Theorem $2.14$, \cite{Roy}]  that $(N_2, N_3)$ is a $\Gamma$-isometry.
		
Since $R^*_1R_1 = I - \frac{1}{4}R^*_2R_2$, we note that
\begin{equation}\label{U 7}
			\begin{aligned}
				R^*_1R_1&=\begin{pmatrix}
					P^*_1 & B^*_1\\
					0 & N^*_1
				\end{pmatrix}
				\begin{pmatrix}
					P_1 & 0\\
					B_1 & N_1
				\end{pmatrix} \\&=\begin{pmatrix}
					P^*_1P_1 + B^*_1B_1 & B^*_1N_1\\
					N^*_1B_1 & N^*_1N_1
				\end{pmatrix} \\&=I - \frac{1}{4}R^*_2R_2\\&=
				\begin{pmatrix}
					I_{\mathcal{H}} & 0\\
					0 & I_{\mathcal{K} \ominus \mathcal{H}}
				\end{pmatrix} -
				\frac{1}{4}
				\begin{pmatrix}
					P^*_2 & B^*_2\\
					0 & N^*_2
				\end{pmatrix}
				\begin{pmatrix}
					P_2 & 0\\
					B_2 & N_2
				\end{pmatrix}\\&=\begin{pmatrix}
					I_{\mathcal{H}} - \frac{1}{4}(P^*_2P_2 + B^*_2B_2) & -\frac{B^*_2N_2}{4}\\
					-\frac{N^*_2B_2}{4} & I_{\mathcal{K} \ominus \mathcal{H}} - \frac{1}{4}N^*_2N_2
				\end{pmatrix}
			\end{aligned}
		\end{equation}
It follows from \eqref{U 7} that
		\begin{equation}\label{U 8}
			\begin{aligned}
				P^*_1P_1 + B^*_1B_1 &= I_{\mathcal{H}} - \frac{1}{4}(P^*_2P_2 + B^*_2B_2), \,\, N^*_1B_1 = - \frac{1}{4}N^*_2B_2, \,\, N^*_1N_1 = I_{\mathcal{K} \ominus \mathcal{H}} - \frac{1}{4}N^*_2N_2.
			\end{aligned}
		\end{equation}
		Since $(N_2, N_3)$ is a $\Gamma$-isometry and $N^*_1N_1 = I_{\mathcal{K} \ominus \mathcal{H}} - \frac{1}{4}N^*_2N_2,$  it yields  from [Theorem $5.2$, \cite{Jindal}]  that $\textbf{N} = (N_1, N_2, N_3)$ is a $\mathcal{\bar{P}}$-isometry. Thus, by definition, $\textbf{N}$ can be extended to a $\mathcal{\bar{P}}$-unitary $\textbf{U} = (U_1, U_2, U_3)$ on some larger Hilbert space. Hence, by definition of $\mathcal{\bar{P}}$-unitary, we conclude that the Taylor joint spectrum $\sigma(\textbf{U})$ is contained in the distinguished boundary $b\overline{\mathcal{P}}$ of$\mathcal{\bar{P}}$,  so $(1)$ follows.
		
Since  $R_2R_3=R_3R_2,$  we have
		\begin{equation}\label{U 9}
			\begin{aligned}
				B_2P_3 - B_3P_2 &= N_3B_2 - N_2B_3,
			\end{aligned}
		\end{equation}
We see from \eqref{U 5} and  \eqref{U 11} that 		
\begin{equation}\label{U 91}
			\begin{aligned}
				B_2P_3 - B_3P_2 &= N^*_2B_3P_3-\Theta D_{P_3}P_2\\&=N^*_2\Theta D_{P_3}P_3-\Theta D_{P_3}P_2
			\end{aligned}
		\end{equation}		
$${\rm{and}}$$
\begin{equation}\label{U 92}
			\begin{aligned}
				N_3B_2 - N_2B_3 &= N_3N^*_2B_3-N_2\Theta D_{P_3}\\&=(N_3N_2^*-N_2)\Theta D_{P_3}.
			\end{aligned}
		\end{equation}
We deduce from \eqref{U 9}, \eqref{U 91} and \eqref{U 92} that 		
\begin{equation}\label{U 93}
N^*_2\Theta D_{P_3}P_3-\Theta D_{P_3}P_2=(N_3N_2^*-N_2)\Theta D_{P_3}.
\end{equation}		
Multiplying $\Theta^*$ from left side of \eqref{U 93} and by using \eqref{U 6}, we conclude that 		
		\begin{equation}\label{U 10}
			\begin{aligned}
				(X D_{P_3}P_3 -  D_{P_3}P_2) &= \Theta^*(N_3N^*_2 - N_2)\Theta D_{P_3}.
			\end{aligned}
		\end{equation}
Therefore, it follows from \eqref{U 10}  that $(X D_{P_3}P_3 -  D_{P_3}P_2)|_{Ker D_{P_3}} = 0$. This completes the proof.
	\end{proof}
	
%	There is no result on conditional dilation for$\mathcal{\bar{P}}$-contraction in literature. But there is an unpublished work \cite{SP} by S. Pal, where he developed a conditional dilation for a$\mathcal{\bar{P}}$-contraction. The details can be found in Theorem 8.5 and Theorem 8.6 of \cite{SP}.
	
	%\begin{rem}
		%Note that, left multiplying by $\Theta^*$ both side we get
		%\begin{equation*}
			%\begin{aligned}
				%&\Theta^*A^*_2\Theta D_{P_3}P_3 - \Theta^*\Theta D_{P_3}P_2 =
				%\Theta^*(A_3A^*_2 - A_2)\Theta D_{P_3}\\
				%\Rightarrow&
				%X^*D_{P_3}P_3 - XD_{P_3} - X^*D_{P_3}P_3 = \Theta^*(A_3A^*_2 - A_2)\Theta D_{P_3}
			%\end{aligned}
		%\end{equation*}
		%Using $(4.5)$ we get $\Theta^*A_3A^*_2\Theta D_{P_3} = 0$.
	%\end{rem}
	
	\section{Some Special Forms of $\Gamma_{E(3; 3; 1, 1, 1)}$-Contraction and $\Gamma_{E(3; 2; 1, 2)}$-Contraction}\label{Sec 4}
In this section we discuss $\Gamma_{E(3; 3; 1, 1, 1)} $-contractions $\textbf{T} = (T_1, \dots, T_7)$ and $\Gamma_{E(3; 2; 1, 2)} $-contractions $\textbf{S} = (S_1, S_2, S_3, \tilde{S}_1, \tilde{S}_2)$, where $T_7$ and $S_3$ are partial isometries, to provide more examples for analysis. We only state the following lemma from [Lemma $3.1$, \cite{Ball}].	\begin{lem}\label{Lemma Partial}
		Let $T$ be a contraction on a Hilbert space $\mathcal{H}$. Then $T$ is a partial isometry if and only if $\mathcal{H}$ can be decomposed as $\mathcal{H} = \mathcal{H}_1 \oplus \mathcal{H}_2$ such that
		\begin{equation*}
			\begin{aligned}
				T = 
				\begin{bmatrix}
					Z & 0
				\end{bmatrix} : \mathcal{H}_1 \to \mathcal{H}
			\end{aligned}
		\end{equation*}
		for some isometry $Z : \mathcal{H}_1 \to \mathcal{H}$.
	\end{lem}
	
	\begin{prop}\label{Prop Partial 1}
		Let $\textbf{T} = (T_1, \dots, T_7)$ be a $\Gamma_{E(3; 3; 1, 1, 1)}$-contraction on a Hilbert space $\mathcal{H}$, with $T_7$ being a partial isometry. Suppose that $F_1, \dots, F_6$ are fundamental operators for $\textbf{T}$. Then the following is true:
\begin{enumerate}
\item ${\rm{Ker}}~ T_7$ is jointly invariant under $(T_1, \dots, T_6),$ and
\item if we denote $(D_1, \dots, D_6) = (T_1 \dots, T_7)|_{Ker T_7}$, then 
\begin{enumerate}
\item  $F_iF_j = F_jF_i$ if and only if $D_iD_j=D_jD_i$ for $1 \leqslant i, j \leqslant 6$,
\item $F_iF^*_j - F^*_jF_i = F_jF^*_i - F^*_iF_j$ if and only if $D_iD ^*_j - D^*_jD_i = D_jD ^*_i - D ^*_iD_j$ for $1 \leqslant i, j \leqslant 6$.
\end{enumerate}
\end{enumerate}
	\end{prop}
	
	\begin{proof}
We first note that $T_7$, being a partial isometry, from Lemma \ref{Lemma Partial}, we get
\begin{equation}\label{Partial}
\begin{aligned}
D^2_{T_7} &=
\begin{pmatrix}
I_{\Ran T^*_7} & 0\\
0 & I_{Ker T_7}
\end{pmatrix} -
\begin{pmatrix}
I_{\Ran T^*_7} & 0\\
0 & 0
\end{pmatrix} =
\begin{pmatrix}
0 & 0\\
0 & I_{Ker T_7}
\end{pmatrix}&=D_{T_7}
\end{aligned}
\end{equation}		
It implies from \eqref{Partial} that $\mathcal{D}_{T_7} = \{0\} \oplus Ker T_7$. Thus, the fundamental operators $F_i, 1\leq i \leq 6,$ acting on $\mathcal{D}_{T_7},$ are expressed as follows:
		\begin{equation}\label{Partial 1}
			\begin{aligned}
				F_i &=
				\begin{pmatrix}
					0 & 0\\
					0 & P_i
				\end{pmatrix} \,\, \text{for} \,\, 1 \leqslant i \leqslant 6
			\end{aligned}
		\end{equation}
for some $P_i, 1\leq i \leq 6$ on $Ker T_7$.
Let the $2 \times 2$ block matrix of $T_i, 1\leq i \leq 6$ be the form 		
\begin{equation}\label{Partial 2}
			\begin{aligned}
				T_i &=
				\begin{pmatrix}
					A_i & B_i\\
					C_i & D_i
				\end{pmatrix} \,\, \text{for} \,\, 1 \leqslant i \leqslant 6
			\end{aligned}
		\end{equation} with respect the decomposition $\mathcal{H} = \Ran T^*_7 \oplus Ker T_7$ of $\mathcal{H}$ 		and 
		\begin{equation}\label{Partial 3}
			\begin{aligned}
				T_7 &=
				\begin{pmatrix}
					X & 0\\
					Y & 0
				\end{pmatrix} : \Ran T^*_7 \oplus Ker T_7 \to \Ran T^*_7 \oplus Ker T_7.
			\end{aligned}
		\end{equation}
Since $T_7$ is a partial isometry, it follows from  Lemma \ref{Lemma Partial} that
		$Z =
		\begin{pmatrix}
			X\\
			Y
		\end{pmatrix}$ is an isometry.  As $\textbf{T} = (T_1, \dots, T_7)$ is a $\Gamma_{E(3; 3; 1, 1, 1)}$-contraction, it yields from [Theorem $2.4$, \cite{apal3}] that there exists unique operators $F_i$ and $F_{7-i}$ in $\mathcal{B}(\mathcal{D}_{T_7})$ for $1\leq i \leq 6$ such that the operator $F_i + zF_{7-i}, 1\leq i \leq 6,$ has numerical radius not exceeding $1$ for every $z \in \mathbb{T}$ and 
\begin{equation}\label{Partial 4}
			\begin{aligned}
				T_i - T^*_{7-i}T_7 &=
				D_{T_7}F_iD_{T_7} \,\,\text{and}\,\,
				T_{7-i} - T^*_iT_7 =
				D_{T_7}F_{7-i}D_{T_7}.
			\end{aligned}
		\end{equation}

We notice from \eqref{Partial 2} and \eqref{Partial 4}  that for $1\leq i \leq 6$
\begin{equation}\label{Partial 11}
\begin{aligned}
T_i - T^*_{7-i}T_7 &=\begin{pmatrix}
					A_i & B_i\\
					C_i & D_i
				\end{pmatrix}-\begin{pmatrix}
					A_{7-i }^*& C_{7-i}^*\\
					B_{7-i }^*& D_{7-i}^*
				\end{pmatrix}\begin{pmatrix}
					X & 0\\
					Y & 0
				\end{pmatrix}\\&=\begin{pmatrix}
					A_i-A_{7-i}^*X-C_{7-i}^*Y & B_i\\
					C_i-B_{7-i}^*X-D_{7-i}^*Y & D_i
				\end{pmatrix}\\&=\begin{pmatrix}
					0 & 0\\
					0 & P_i
				\end{pmatrix}
\end{aligned}
\end{equation}
From $(\ref{Partial 11}),$ we derive
		\begin{equation}\label{Partial 5}
			\begin{aligned}
				A_i = (A^*_{7-i}X + C^*_{7-i}Y), B_i = 0, C_i = D^*_{7-i}Y,  D_i = P_i,
			\end{aligned}
		\end{equation}
Therefore, from \eqref{Partial 5}, we deduce that 		
\begin{equation}\label{Partial 7}
			\begin{aligned}
				T_i &=
				\begin{pmatrix}
					A^*_{7-i}X + C^*_{7-i}Y& 0\\
					D^*_{7-i}Y & P_i
				\end{pmatrix} \,\, \text{for} \,\, 1 \leqslant i \leqslant 6.
			\end{aligned}
		\end{equation}
It implies from \eqref{Partial 1} and \eqref{Partial 5} that
		\begin{equation}\label{Partial 8}
			\begin{aligned}
				F_i &=
				\begin{pmatrix}
					0 & 0\\
					0 & D_i
				\end{pmatrix} \,\, \text{for} \,\, 1 \leqslant i \leqslant 6,
			\end{aligned}
		\end{equation}
and $(1)$ and $(2)$ follow. This completes the proof.		
	\end{proof}
We state the  analogous theorem for $\Gamma_{E(3; 2; 1, 2)}$-contraction. It's proof is similar to that of the previous theorem. Therefore, we skip the proof.	
		\begin{prop}\label{Prop Partial 2}
		Let $\textbf{S} = (S_1, S_2, S_3, \tilde{S}_1, \tilde{S}_2)$ be a $\Gamma_{E(3; 2; 1, 2)}$-contraction on a Hilbert space $\mathcal{H}$ with $S_3$ partial isometry and $G_1, 2G_2, 2\tilde{G}_1, \tilde{G}_2$ be the fundamental operators of $\textbf{S}$. Then the following hold:
		\begin{enumerate}
			\item $Ker ~S_3$ is invariant under $(S_1, S_2, \tilde{S}_1, \tilde{S}_2)$, and
			\item if we denote $(E_1, 2E_2, 2\tilde{E}_1, \tilde{E}_2) = (S_1, S_2, \tilde{S}_1, \tilde{S}_2)|_{Ker S_3},$ then
			\begin{enumerate}
				\item  $G_1, 2G_2, 2\tilde{G}_1, \tilde{G}_2$ commute with each other if and only if $E_1, 2E_2, 2\tilde{E}_1, \tilde{E}_2$ commute with each other,
				
				\item $G^*_1G_1 - G_1G^*_1 = 4(G^*_2G_2 - G_2G^*_2)$ if and only if $E^*_1E_1 - E_1E^*_1 = 4(E^*_2E_2 - E_2E^*_2)$,
				
				\item $G^*_1G_1 - G_1G^*_1 = 4(\tilde{G}^*_1\tilde{G}_1 - \tilde{G}_1\tilde{G}^*_1)$ if and only if $E^*_1E_1 - E_1E^*_1 = 4(\tilde{E}^*_1\tilde{E}_1 - \tilde{E}_1\tilde{E}^*_1)$,
				
				\item $G^*_1G_1 - G_1G^*_1 = \tilde{G}^*_2\tilde{G}_2 - \tilde{G}_2\tilde{G}^*_2$ if and only if $E^*_1E_1 - E_1E^*_1 = \tilde{E}^*_2\tilde{E}_2 - \tilde{E}_2\tilde{E}^*_2$,
				
				\item $G^*_2G_2 - G_2G^*_2 = \tilde{G}^*_1\tilde{G}_1 - \tilde{G}_1\tilde{G}^*_1$ if and only if $E^*_2E_2 - E_2E^*_2 = \tilde{E}^*_1\tilde{E}_1 - \tilde{E}_1\tilde{E}^*_1$,
				
				\item $4(G^*_2G_2 - G_2G^*_2) = \tilde{G}^*_2\tilde{G}_2 - \tilde{G}_2\tilde{G}^*_2$ if and only if $4(E^*_2E_2 - E_2E^*_2) = \tilde{E}^*_2\tilde{E}_2 - \tilde{E}_2\tilde{E}^*_2$,
				
				\item $4(\tilde{G}^*_1\tilde{G}_1 - \tilde{G}_1\tilde{G}^*_1) = \tilde{G}^*_2\tilde{G}_2 - \tilde{G}_2\tilde{G}^*_2$ if and only if $4(\tilde{E}^*_1\tilde{E}_1 - \tilde{E}_1\tilde{E}^*_1) = \tilde{E}^*_2\tilde{E}_2 - \tilde{E}_2\tilde{E}^*_2$,
				
				\item $G_1G^*_2 - G^*_2G_1 = G_2G^*_1 - G^*_1G_2$ if and only if $E_1E^*_2 - E^*_2E_1 = E_2E^*_1 - E^*_1E_2$,
				
				\item $G_1\tilde{G}^*_1 - \tilde{G}^*_1G_1 = \tilde{G}_1G^*_1 - G^*_1\tilde{G}_1$ if and only if $E_1\tilde{E}^*_1 - \tilde{E}^*_1E_1 = \tilde{E}_1E^*_1 - E^*_1\tilde{E}_1$,
				
				\item $G_1\tilde{G}^*_2 - \tilde{G}^*_2G_1 = \tilde{G}_2G^*_1 - G^*_1\tilde{G}_2$ if and only if $E_1\tilde{E}^*_2 - \tilde{E}^*_2E_1 = \tilde{E}_2E^*_1 - E^*_1\tilde{E}_2$,
				
				\item $G_2\tilde{G}^*_1 - \tilde{G}^*_1G_2 = \tilde{G}_1G^*_2 - G^*_2\tilde{G}_1$ if and only if $E_2\tilde{E}^*_1 - \tilde{E}^*_1E_2 = \tilde{E}_1E^*_2 - E^*_2\tilde{E}_1$,
				
				\item $G_2\tilde{G}^*_2 - \tilde{G}^*_2G_2 = \tilde{G}_2G^*_2 - G^*_2\tilde{G}_2$ if and only if $E_2\tilde{E}^*_2 - \tilde{E}^*_2E_2 = \tilde{E}_2E^*_2 - E^*_2\tilde{E}_2$,
				
				\item $\tilde{G}_1\tilde{G}^*_2 - \tilde{G}^*_2\tilde{G}_1 = \tilde{G}_2\tilde{G}^*_1 - \tilde{G}^*_1\tilde{G}_2$ if and only if $\tilde{E}_1\tilde{E}^*_2 - \tilde{E}^*_2\tilde{E}_1 = \tilde{E}_2\tilde{E}^*_1 - \tilde{E}^*_1\tilde{E}_2$.
			\end{enumerate}
		\end{enumerate}
	\end{prop}

We only state the following theorem from [Theorem $2.9$, \cite{Mandal}].
\begin{thm}\label{gamma 3 isometry}
Suppose that $(T_1,T_2,V_3)$ is  a commuting $3$-tuple of operators acting on some Hilbert space $\mathcal H$ with $T_1$ and $T_2$ are contractions and $V_3$  is an isometry. Then $\left(\frac{T_1+T_2+V_3}{3},\frac{T_1T_2+T_2V_3+V_3T_1}{3},T_1T_2V_3 \right)$ is a $\Gamma_3$-contraction. Moreover, $\left(\frac{T_1+T_2+V_3}{3},\frac{T_1T_2+T_2V_3+V_3T_1}{3},T_1T_2V_3\right)$ has a $\Gamma_3$-isometric dilation.
\end{thm}
\begin{rem}\label{isso}
We observe from [Theorem $2.9$, \cite{Mandal}] that $$\left(\tilde{V}_1=\frac{V_1+V_2+V_3\oplus I_{\mathcal K\ominus \mathcal H}}{3},\tilde{V}_2=\frac{V_1V_2+V_2(V_3\oplus I_{\mathcal K\ominus \mathcal H})+(V_3\oplus I_{\mathcal K\ominus \mathcal H})V_1}{3},\tilde{V}_3=V_1V_2(V_3\oplus I_{\mathcal K\ominus \mathcal H})\right),$$ is the $\Gamma_3$-isometric dilation of $\left(\frac{T_1+T_2+V_3}{3},\frac{T_1T_2+T_2V_3+V_3T_1}{3},T_1T_2V_3\right)$. Note that $\|\tilde{V}_i\|\leq 1$ for $1\leq i \leq 3.$ It follows from [Thorem $5.7$, \cite{Bhattacharyya}] that $(\tilde{V}_1,\tilde{V}_2,\tilde{V}_3)$ is also a $\Gamma_{E(2; 2; 1, 1)}$-isometry.

\end{rem}	
Let $\textbf{x}=(x_1,x_2,\dots, x_7)$ and  \begin{equation}\label{psi13}
		\begin{aligned}
			\Psi^{(3)}(z, w, \textbf{x})
			&= \frac{x_4 - zx_5 - wx_6 + zwx_7}{1 - zx_1 - wx_2 + zwx_3}, \, z,w \in \mathbb{D}.
		\end{aligned}
	\end{equation}

\begin{lem}\label{tetra}
$(x_1,0,0,0,0,x_6,x_7)\in \Gamma_{E(3; 3; 1, 1, 1)}$ if and only if $(x_1,x_6,x_7) \in \Gamma_{E(2; 2; 1, 1)}$.
\end{lem}	
\begin{proof}
By [Theorem $2.9$, \cite{apal1}], a point $(x_1,0,0,0,0,x_6,x_7)\in \Gamma_{E(3; 3; 1, 1, 1)}$ if and only if $(0,0,\frac{x_6-zx_7}{1-x_1z}) \in 
 \Gamma_{E(2; 2; 1, 1)}$ for all $z\in \mathbb D$. As $(0,0,\frac{x_6-zx_7}{1-x_1z}) \in  \Gamma_{E(2; 2; 1, 1)}$ for all $z \in \mathbb D,$ it implies from [Theorem $2.4$,\cite{Abouhajar}] that $\Big |\frac{x_6-zx_7}{1-x_1z}\Big |\leq 1$ for all $z\in \mathbb D$ and hence by [Theorem $2.4$,\cite{Abouhajar}], we deduce that $(x_1,x_6,x_7) \in \Gamma_{E(2; 2; 1, 1)}$.
 
Conversely, suppose that $(x_1,x_6,x_7) \in \Gamma_{E(2; 2; 1, 1)}$. Then, by [Theorem $2.4$,\cite{Abouhajar}], we get  $\Big |\frac{x_6-zx_7}{1-x_1z}\Big |\leq 1$ for all $z\in \mathbb D.$  By [Theorem $2.8$, \cite{apal1}], a point $(x_1,0,0,0,0,x_6,x_7)\in \Gamma_{E(3; 3; 1, 1, 1)}$ if and only if $(x_1,0,0)\in G_{E(2;2;1,1)}$ and $$\|\Psi^{(3)}(\cdot,(x_1,0,0,0,0,x_6,x_7)\|_{H^{\infty}(\mathbb{D}^{2})}\leq 1.$$
As $(x_1,x_6,x_7)\in \Gamma_{E(2; 2; 1, 1)}$, we have $1-x_1z\neq 0$ for all $z \in \mathbb D,$ which implies that $(x_1,0,0)\in G_{E(2;2;1,1)}$. We notice from \eqref{psi13} that for all $z,w \in \mathbb D$
\begin{equation}\label{psi23}
\begin{aligned}
|\Psi^{(3)}(z, w, (x_1,0,0,0,0,x_6,x_7)|
			&= \Big|\frac{ - wx_6 + zwx_7}{1 - zx_1 }\Big|\\&<\Big|\frac{x_6-zx_7}{1-x_1z}\Big| \\&\leq 1
\end{aligned}
\end{equation}
It follows from \eqref{psi23} that $$\|\Psi^{(3)}(\cdot,(x_1,0,0,0,0,x_6,x_7)\|_{H^{\infty}(\mathbb{D}^{2})}\leq 1$$  and hence by above observations, we conclude that $(x_1,0,0,0,0,x_6,x_7)\in \Gamma_{E(3; 3; 1, 1, 1)}$. This completes the proof.

\end{proof}
\begin{rem}
By using a similar argument, one can show that $(0,x_2,0,0,x_5,0,x_7)\in \Gamma_{E(3; 3; 1, 1, 1)}$ if and only if $(x_2,x_5,x_7) \in \Gamma_{E(2; 2; 1, 1)}$. We can also demonstrate that $(0,0,x_3,x_4,0,0,x_7)\in \Gamma_{E(3; 3; 1, 1, 1)}$ if and only if $(x_3,x_4,x_7) \in \Gamma_{E(2; 2; 1, 1)}$.
\end{rem}	
In the following proposition, we establish a relationship between $\Gamma_{E(3; 3; 1, 1, 1)}$-isometry and $\Gamma_{E(2; 2; 1, 1)}$-isometry.
	
	\begin{prop}\label{isso1}
		Let $(T_1,T_6, T_7)$ be a commuting triple of bounded operators on a Hilbert space $\mathcal{H}$. Then $(T_1, T_6, T_7)$ is a $\Gamma_{E(2; 2; 1, 1)}$-isometry if and only if  $(T_1, 0, 0, 0, 0, T_6, T_7)$ is a $\Gamma_{E(3; 3; 1, 1, 1)}$-isometry.
			\end{prop}
	
	\begin{proof}
Suppose that  $(T_1, T_6, T_7)$ is a $\Gamma_{E(2; 2; 1, 1)}$-isometry. It follows from [Thorem $5.7$, \cite{Bhattacharyya}] that
$(T_1, T_6, T_7)$ is a $\Gamma_{E(2; 2; 1, 1)}$-contraction and $T_7$ is an isometry. Define the map $\varphi : \mathbb{C}^3 \to \mathbb{C}^7$ by \[\varphi(z_1, z_6, z_7) = (z_1, 0, 0, 0, 0, z_6, z_7).\]
We observe that for any $p \in \mathbb{C}[z_1, \dots, z_7],$ we have $p \circ \varphi \in \mathbb{C}[z_1, z_6, z_7].$ Thus, we have
		\begin{equation*}
			\begin{aligned}
				\|p(T_1, 0, 0, 0, 0, T_6, T_7)\|
				&=
				\|p \circ \varphi(T_1, T_6, T_7)\|\\
				&\leqslant
				\|p \circ \varphi\|_{\infty, \Gamma_{E(2; 2; 1, 1)}}\\
				&=
				\|p\|_{\infty, \varphi(\Gamma_{E(2; 2; 1, 1)})}\\
				&\leqslant
				\|p\|_{\infty, \Gamma_{E(3; 3; 1, 1, 1)}}.
			\end{aligned}
		\end{equation*}
This shows that $(T_1, 0, 0, 0, 0, T_6, T_7)$ is a $\Gamma_{E(3; 3; 1, 1, 1)}$-contraction. As $(T_1, 0, 0, 0, 0, T_6, T_7)$ is a $\Gamma_{E(3; 3; 1, 1, 1)}$-contraction and $T_7$ is an isometry, it yields from [Theorem $4.4$, \cite{apal2}] that  $(T_1, 0, 0, 0, 0, T_6, T_7)$ is a $\Gamma_{E(3; 3; 1, 1, 1)}$-isometry.

Conversely, suppose that  $(T_1, 0, 0, 0, 0, T_6, T_7)$ is a $\Gamma_{E(3; 3; 1, 1, 1)}$-isometry. Then by [Theorem $4.4$, \cite{apal2}], we conclude that $(T_1, T_6, T_7)$ is a $\Gamma_{E(2; 2; 1, 1)}$-isometry. This completes the proof.	
\end{proof}
\begin{rem}
By using a similar argument, one can easily prove that $(0,T_2, 0, 0, T_5, 0,T_7)$ is a $\Gamma_{E(3; 3; 1, 1, 1)}$-isometry if and only if $(T_2,T_5,T_7) \in \Gamma_{E(2; 2; 1, 1)}$-isometry. We can also show that $(T_3, T_4, T_7)$ is a $\Gamma_{E(2; 2; 1, 1)}$-isometry if and only if  $(0, 0, T_3, T_4, 0, 0, T_7)$ is a $\Gamma_{E(3; 3; 1, 1, 1)}$-isometry.\end{rem}	
We will now produce an example of  $\Gamma_{E(3; 3; 1, 1, 1)}$-contraction  which satisfies all conditions in Proposition \ref{Prop Partial 1}. The following example is found in section $2$ in \cite{Mandal}.
\begin{ex}\label{example1} Let $H^2(\mathbb D)$  denotes the Hardy space over the unit disc $\mathbb D.$ Consider the following triple of commuting operators on $H^2(\mathbb D)\oplus H^2(\mathbb D):$ 
$$(A,B,P)=\left( \left(\begin{smallmatrix} 0 & 0 \\I_{H^2} & 0 \end{smallmatrix}\right), \left(\begin{smallmatrix} M_z & 0 \\  0 & M_z \end{smallmatrix}\right), \left(\begin{smallmatrix} I_{H^2} & 0 \\ 0 & I_{H^2}\end{smallmatrix}\right)\right),$$ where $M_z$ is a multiplication operator on $H^2.$ Clearly, $I-M_z^*M_z=0.$  Let 
$$T_1=\frac{1}{3}(A+B+P)=\frac{1}{3}\left(\begin{smallmatrix} I_{H^2}+M_z & 0 \\ I_{H^2} & I_{H^2}+M_z \end{smallmatrix}\right), T_2=0=T_3=T_4=T_5, T_6=\frac{1}{3}(AB+BP+AP)=\frac{1}{3}\left(\begin{smallmatrix} M_z & 0\\I_{H^2}+M_z  & M_z \end{smallmatrix}\right)$$ and $ T_7=ABP=\left(\begin{smallmatrix} 0 & 0 \\ M_z & 0  \end{smallmatrix}\right).$ By Remark \ref{isso}, we conclude that 
$(T_1, T_6, T_7)$ is a $\Gamma_{E(2; 2; 1, 1)}$-isometry and hence it follows from Proposition \ref{isso1} that $(T_1, 0, 0, 0, 0, T_6, T_7)$ is a $\Gamma_{E(3; 3; 1, 1, 1)}$-isometry.  Note that $$D^2_{T_7}=\left(\begin{smallmatrix} I_{H^2} & 0 \\ 0 & I_{H^2} \end{smallmatrix}\right)-\left(\begin{smallmatrix} 0 & 0 \\ M_z & 0  \end{smallmatrix}\right)^*\left(\begin{smallmatrix} 0 & 0 \\ M_z & 0  \end{smallmatrix}\right)=\left(\begin{smallmatrix} 0 & 0 \\  0 & I_{H^2} \end{smallmatrix}\right)=D_{T_7}.$$ 
Let us consider $$(F_1,F_2,F_3,F_4,F_5,F_6)=(\frac{I_{H^2}+M_z}{3},0,0,0,0,\frac{M_z}{3}).$$  One can easily check that all the conditions of the Proposition \ref{Prop Partial 1} are satisfied.

\end{ex}
We produce an example of a $\Gamma_{E(3; 3; 1, 1, 1)}$-contraction that possesses a $\Gamma_{E(3; 3; 1, 1, 1)}$-isometric dilation but the condition $(2)(b)$ in Proposition \ref{Prop Partial 1} is not fulfilled, namely, $[F_{7-i}^*, F_j] \ne [F_{7-j}^*, F_i] $ for some $i,j$ with $1\leq i ,j\leq 6.$ In summary, we  conclude that the set of sufficient conditions for the existence of a $\Gamma_{E(3; 3; 1, 1, 1)}$-isometric dilation presented in Theorem \ref{conddilation} are generally not necessary, even when the $\Gamma_{E(3; 3; 1, 1, 1)}$-contraction $\textbf{T}=(T_1,\dots, T_7)$ has a special form, where $T_7$ is a partial isometry on $\mathcal H.$
\begin{ex}\label{Exam 1}
Let $\mathcal{H} = H^2(\mathbb{D}) \oplus H^2(\mathbb{D}) \oplus H^2(\mathbb{D})$ and $T_1,T_2$ be two operators defined by
		\begin{equation*}
			\begin{aligned}
				T_1 &=
				\begin{pmatrix}
					0 & I_{H^2} & 0\\
					0 & 0 & I_{H^2}\\
					0 & 0 & 0
				\end{pmatrix}
				~\text{and}~
				T_2 =
				\begin{pmatrix}
					M_z & 0 & 0\\
					0 & M_z & 0\\
					0 & 0 & M_z
				\end{pmatrix}.
			\end{aligned}
		\end{equation*}
Clearly, $T_1$ and $T_2$ are commuting contractions on $\mathcal H.$ By Theorem \ref{contr}, we conclude that the $7$-tuple of operators $\textbf{T} = (T_1, T_2, T_1T_2, T_1T_2, T_1^2T_2, T_1T_2^2, T_1^2T_2^2)$  has $\Gamma_{E(3; 3; 1, 1, 1)}$-isometric dilation. Note that \begin{equation}
			\begin{aligned}
				T_1 &=
				\begin{pmatrix}
					0 & I_{H^2} & 0\\
					0 & 0 & I_{H^2}\\
					0 & 0 & 0
				\end{pmatrix},
				T_2 = \begin{pmatrix}
					M_z & 0 & 0\\
					0 & M_z & 0\\
					0 & 0 & M_z
				\end{pmatrix},
				T_3 =T_1T_2= \begin{pmatrix}
					0 & M_z & 0\\
					0 & 0 & M_z\\
					0 & 0 & 0
				\end{pmatrix},
				T_4=T_1T_2 = \begin{pmatrix}
					0 & M_z & 0\\
					0 & 0 & M_z\\
					0 & 0 & 0
				\end{pmatrix},\\
				&\hspace{2cm}
				T_5 =T_1^2T_2= \begin{pmatrix}
					0 & 0 & M_z\\
					0 & 0 & 0\\
					0 & 0 & 0
				\end{pmatrix},
				T_6=T_1T_2^2 = \begin{pmatrix}
					0 & M_z^2 & 0\\
					0 & 0 & M_z^2\\
					0 & 0 & 0
				\end{pmatrix},
				T_7 =T_1^2T_2^2= \begin{pmatrix}
					0 & 0 & M_z^2\\
					0 & 0 & 0\\
					0 & 0 & 0
				\end{pmatrix}.
			\end{aligned}
		\end{equation}
Since $M_z$ is an isometry, it implies that $M_z^2$ is also an isometry. Since $M_z^2$ is an isometry, it follows that  $T_7$ is a partial isometry. Observe that \begin{equation}
			\begin{aligned}
				D^2_{T_7} = I - T^*_7T_7 &=
				\begin{pmatrix}
					I_{H^2} & 0 & 0\\
					0 & I_{H^2} & 0\\
					0 & 0 & I_{H^2}
				\end{pmatrix} -
				\begin{pmatrix}
					0 & 0 & 0\\
					0 & 0 & 0\\
					M_z^{*2} & 0 & 0
				\end{pmatrix}\begin{pmatrix}
					0 & 0 & M_z^2\\
					0 & 0 & 0\\
					0 & 0 & 0
				\end{pmatrix} = 
				\begin{pmatrix}
					I_{H^2} & 0 & 0\\
					0 & I_{H^2} & 0\\
					0 & 0 & 0
				\end{pmatrix} = D_{T_7}.
			\end{aligned}
		\end{equation}
Let us set  \begin{equation}
			\begin{aligned}
				\left(F_1F_2,F_3,F_4,F_5,F_6\right)=\left(\begin{pmatrix}
					0 & I_{H^2}\\
					0 & 0
				\end{pmatrix},\begin{pmatrix}
					M_z & 0\\
					0 & M_z
				\end{pmatrix},\begin{pmatrix}
					0 & M_z\\
					0 & 0
				\end{pmatrix},
				\begin{pmatrix}
					0 & M_z\\
					0 & 0
				\end{pmatrix},
			\begin{pmatrix}
					0 & 0\\
					0 & 0
				\end{pmatrix},
				\begin{pmatrix}
					0 & M_z^2\\
					0 & 0
				\end{pmatrix}\right).
			\end{aligned}
		\end{equation}
		Notice that
		\begin{equation*}
			\begin{aligned}
				T_1 - T^*_6T_7 &=
				\begin{pmatrix}
					0 & I_{H^2} & 0\\
					0 & 0 & I_{H^2}\\
					0 & 0 & 0
				\end{pmatrix} -
				\begin{pmatrix}
					0 & 0 & 0\\
					M^{*2}_z & 0 & 0\\
					0 & M^{*2}_z & 0
				\end{pmatrix}
				\begin{pmatrix}
					0 & 0 & M^2_z\\
					0 & 0 & 0\\
					0 & 0 & 0
				\end{pmatrix}\\
				&=
				\begin{pmatrix}
				0 & I_{H^2} & 0\\
				0 & 0 & 0\\
				0 & 0 & 0
				\end{pmatrix}\\&
				= \begin{pmatrix}
				I_{H^2} & 0 & 0\\
				0 & I_{H^2} & 0\\
				0 & 0 & 0
				\end{pmatrix}
				\begin{pmatrix}
				0 & I_{H^2} & 0\\
				0 & 0 & 0\\
				0 & 0 & 0
				\end{pmatrix}
				\begin{pmatrix}
				I_{H^2} & 0 & 0\\
				0 & I_{H^2} & 0\\
				0 & 0 & 0
				\end{pmatrix}\\
				&=
				D_{T_7}F_1D_{T_7},
			\end{aligned}
		\end{equation*}
		\begin{equation*}
			\begin{aligned}
				T_6 - T^*_1T_7 &=
				\begin{pmatrix}
					0 & M^2_z & 0\\
					0 & 0 & M^2_z\\
					0 & 0 & 0
				\end{pmatrix} -
				\begin{pmatrix}
					0 & 0 & 0\\
					I_{H^2} & 0 & 0\\
					0 & I_{H^2} & 0
				\end{pmatrix}
				\begin{pmatrix}
					0 & 0 & M^2_z\\
					0 & 0 & 0\\
					0 & 0 & 0
				\end{pmatrix}\\
				&=
				\begin{pmatrix}
					0 & M^2_z & 0\\
					0 & 0 & 0\\
					0 & 0 & 0
				\end{pmatrix}\\&
				= \begin{pmatrix}
					I_{H^2} & 0 & 0\\
					0 & I_{H^2} & 0\\
					0 & 0 & 0
				\end{pmatrix}
				\begin{pmatrix}
					0 & M^2_z & 0\\
					0 & 0 & 0\\
					0 & 0 & 0
				\end{pmatrix}
				\begin{pmatrix}
					I_{H^2} & 0 & 0\\
					0 & I_{H^2} & 0\\
					0 & 0 & 0
				\end{pmatrix}\\
				&=
				D_{T_7}F_6D_{T_7}.
			\end{aligned}
		\end{equation*}
%		\begin{equation*}
%			\begin{aligned}
%				T_2 - T^*_5T_7 &=
%				\begin{pmatrix}
%					M_z & 0 & 0\\
%					0 & M_z & 0\\
%					0 & 0 & M_z
%				\end{pmatrix} -
%				\begin{pmatrix}
%					0 & 0 & 0\\
%					0 & 0 & 0\\
%					M^*_z & 0 & 0
%				\end{pmatrix}
%				\begin{pmatrix}
%					0 & 0 & M^2_z\\
%					0 & 0 & 0\\
%					0 & 0 & 0
%				\end{pmatrix}\\
%				&=
%				\begin{pmatrix}
%					M_z & 0 & 0\\
%					0 & M_z & 0\\
%					0 & 0 & 0
%				\end{pmatrix}
%				= \begin{pmatrix}
%					I_{H^2} & 0 & 0\\
%					0 & I_{H^2} & 0\\
%					0 & 0 & 0
%				\end{pmatrix}
%				\begin{pmatrix}
%					M_z & 0 & 0\\
%					0 & M_z & 0\\
%					0 & 0 & 0
%				\end{pmatrix}
%				\begin{pmatrix}
%					I_{H^2} & 0 & 0\\
%					0 & I_{H^2} & 0\\
%					0 & 0 & 0
%				\end{pmatrix}\\
%				&=
%				D_{T_7}F_2D_{T_7},
%			\end{aligned}
%		\end{equation*}
Similarly we can also show that $$T_2 - T^*_5T_7=D_{T_7}F_2D_{T_7}, T_5 - T^*_2T_7=D_{T_7}F_5D_{T_7},T_3 - T^*_4T_7=D_{T_7}F_3D_{T_7}~{\rm{and}}~~T_4 - T^*_3T_7=D_{T_7}F_4D_{T_7}.$$
One can also easily  check that $F_iF_j=F_jF_i$ for $1\leq i,j\leq 6.$  We observe that
		\begin{equation*}
			\begin{aligned}
				[F^*_1, F_1] =
				\begin{pmatrix}
					- I_{H^2} & 0\\
					0 & I_{H^2}
				\end{pmatrix} &\ne
				\begin{pmatrix}
					- M^2_zM^{*2}_z & 0\\
					0 & I_{H^2}
				\end{pmatrix} = [F^*_6, F_6],
			\end{aligned}
		\end{equation*}
		\begin{equation*}
			\begin{aligned}
				[F^*_2, F_2] =
				\begin{pmatrix}
					I_{H^2} - M_zM^*_z & 0\\
					0 & I_{H^2} - M_zM^*_z
				\end{pmatrix} &\ne
				\begin{pmatrix}
					0 & 0\\
					0 & 0
				\end{pmatrix} = [F^*_5, F_5]
			\end{aligned}
		\end{equation*}
		$${\rm{and}}$$
		\begin{equation*}
			\begin{aligned}
				[F^*_3, F_3] =
				\begin{pmatrix}
					- M_zM^*_z & 0\\
					0 & I_{H^2}
				\end{pmatrix} &=
			 	[F^*_4, F_4].
			\end{aligned}
		\end{equation*}
This implies that the condition in $(2)(b)$ in Proposition \ref{Prop Partial 1} is not satisfied, namely, $[F_{7-i}^*, F_j] \ne [F_{7-j}^*, F_i] $ for some $i,j$ with $1\leq i ,j\leq 6.$ Thus, we deduce that the set of sufficient conditions for the existence of a $\Gamma_{E(3; 3; 1, 1, 1)}$-isometric dilation presented in Theorem \ref{conddilation} are  not necessary in general.
	\end{ex}
We only state the following lemma. It's proof is similar to that of the Lemma \ref{tetra}. Therefore, we skip the proof.	
\begin{lem}
$\textbf{x}=(x_1,0,x_3,0,y_2)\in \Gamma_{E(3;2;1,2)}$ if and only if $(x_1,y_2,x_3)\in \Gamma_{E(2; 2; 1, 1)}$.
\end{lem}	
\begin{rem}
By using a similar argument, one can show that $(0,x_2,x_3,y_1,0)\in \Gamma_{E(3; 3; 1, 1, 1)}$ if and only if $(\frac{x_2}{2},\frac{y_1}{2},x_3) \in \Gamma_{E(2; 2; 1, 1)}$.
\end{rem}
We only state the following Proposition. The proof is analogous to that of Proposition \ref{isso1}. Consequently, we omit the proof.

\begin{prop}
		Let $(S_1, \tilde{S}_2, S_3)$ be a triple of commuting bounded operators on a Hilbert space $\mathcal{H}$. Then
		 $(S_1, \tilde{S}_2, S_3)$ is a $\Gamma_{E(2; 2; 1, 1)}$-isometry if and only if $(S_1, 0, S_3, 0, \tilde{S}_2)$ is a $\Gamma_{E(3; 2; 1, 2)}$-isometry.
\end{prop}
\begin{rem}
By using a similar argument, one can easily demonstrate that $\left(\frac{S_2}{2}, \frac{\tilde{S}_1}{2}, S_3\right)$ is a $\Gamma_{E(2; 2; 1, 1)}$-isometry if and only if $(0, S_2, S_3, \tilde{S}_1, 0)$ is a $\Gamma_{E(3; 2; 1, 2)}$-isometry.
\end{rem}
We discuss an example of a $\Gamma_{E(3; 2; 1, 2)}$-contraction that possesses a $\Gamma_{E(3; 2; 1, 2)}$-isometric dilation by which one of the conditions in $(2)$ of the Proposition \ref{Prop Partial 1} is not satisfied. As a result, we conclude that the set of sufficient conditions for the existence of a $\Gamma_{E(3; 2; 1, 2)}$-isometric dilation presented in Theorem \ref{condilation1} are generally not necessary, even when the $\Gamma_{E(3; 2; 1, 2)}$-contraction $\textbf{S} = (S_1, S_2, S_3, \tilde{S}_1, \tilde{S}_2)$ has a special form, where $S_3$ is a partial isometry on $\mathcal H.$

\begin{ex}\label{Exam 2}
		Consider the Hilbert space $\mathcal{H}$, along with the operators $T_1$ and $T_2$ as discussed in Example \ref{Exam 1}.  It follows from Theorem \ref{Thm Exam 2} that $$\textbf{S} = (S_1, S_2, S_3, \tilde{S}_1, \tilde{S}_2) = (T_1, T_1T_2 + T_1^2T_2, T_1^2T_2^2, T_2 + T_1T_2, T_1T_2^2)$$ is a $\Gamma_{E(3; 2; 1, 2)}$-isometry. Note that
		\small{\begin{equation}
			\begin{aligned}
				&S_1 =
				\left(\begin{smallmatrix}
					0 & I_{H^2} & 0\\
					0 & 0 & I_{H^2}\\
					0 & 0 & 0
				\end{smallmatrix}\right),
				S_2 =
				\left(\begin{smallmatrix}
					0 & M_z & M_z\\
					0 & 0 & M_z\\
					0 & 0 & 0
				\end{smallmatrix}\right),
				S_3 =
				\left(\begin{smallmatrix}
					0 & 0 & M^2_z\\
					0 & 0 & 0\\
					0 & 0 & 0
				\end{smallmatrix}\right),
				\tilde{S}_1 =
				\left(\begin{smallmatrix}
					M_z & M_z & 0\\
					0 & M_z & M_z\\
					0 & 0 & M_z
				\end{smallmatrix}\right) ~{\rm{and}}~
				\tilde{S}_2 =
				\left(\begin{smallmatrix}
					0 & M^2_z & 0\\
					0 & 0 & M^2_z\\
					0 & 0 & 0
				\end{smallmatrix}\right).
			\end{aligned}
		\end{equation}}
		
Clearly, $S_3$ is a partial isometry. Observe that \begin{equation}
			\begin{aligned}
				D^2_{S_3} = I - S^*_3S_3 &=
				\begin{pmatrix}
					I_{H^2} & 0 & 0\\
					0 & I_{H^2} & 0\\
					0 & 0 & I_{H^2}
				\end{pmatrix} -
				\begin{pmatrix}
					0 & 0 & 0\\
					0 & 0 & 0\\
					M^{*2}_z & 0 & 0
				\end{pmatrix}\begin{pmatrix}
					0 & 0 & M^2_z\\
					0 & 0 & 0\\
					0 & 0 & 0
				\end{pmatrix} = 
				\begin{pmatrix}
					I_{H^2} & 0 & 0\\
					0 & I_{H^2} & 0\\
					0 & 0 & 0
				\end{pmatrix} = D_{S_3}.
			\end{aligned}
		\end{equation}
Let us set  \begin{equation}
			\begin{aligned}
				(G_1, 2G_2, 2\tilde{G}_1, \tilde{G}_2)=\left(
				\begin{pmatrix}
					0 & I_{H^2}\\
					0 & 0
				\end{pmatrix},
				\begin{pmatrix}
					0 & M_z\\
					0 & 0
				\end{pmatrix},
				\begin{pmatrix}
					M_z & M_z\\
					0 & M_z
				\end{pmatrix},
				\begin{pmatrix}
					0 & M^2_z\\
					0 & 0
				\end{pmatrix}\right).
			\end{aligned}
		\end{equation}
		Observe that
		\begin{equation*}
			\begin{aligned}
				S_1 - \tilde{S}^*_2S_3 &=
				\begin{pmatrix}
					0 & I_{H^2} & 0\\
					0 & 0 & I_{H^2}\\
					0 & 0 & 0
				\end{pmatrix} -
				\begin{pmatrix}
					0 & 0 & 0\\
					M^{*2}_z & 0 & 0\\
					0 & M^{*2}_z & 0
				\end{pmatrix}
				\begin{pmatrix}
					0 & 0 & M^2_z\\
					0 & 0 & 0\\
					0 & 0 & 0
				\end{pmatrix}\\
				&=
				\begin{pmatrix}
					0 & I_{H^2} & 0\\
					0 & 0 & 0\\
					0 & 0 & 0
				\end{pmatrix}
				\\&= \begin{pmatrix}
					I_{H^2} & 0 & 0\\
					0 & I_{H^2} & 0\\
					0 & 0 & 0
				\end{pmatrix}
				\begin{pmatrix}
					0 & I_{H^2} & 0\\
					0 & 0 & 0\\
					0 & 0 & 0
				\end{pmatrix}
				\begin{pmatrix}
					I_{H^2} & 0 & 0\\
					0 & I_{H^2} & 0\\
					0 & 0 & 0
				\end{pmatrix}\\
				&=
				D_{S_3}G_1D_{S_3},
			\end{aligned}
		\end{equation*}
		\begin{equation*}
			\begin{aligned}
				\tilde{S}_2 - S^*_1S_3 &=
				\begin{pmatrix}
					0 & M^2_z & 0\\
					0 & 0 & M^2_z\\
					0 & 0 & 0
				\end{pmatrix} -
				\begin{pmatrix}
					0 & 0 & 0\\
					I_{H^2} & 0 & 0\\
					0 & I_{H^2} & 0
				\end{pmatrix}
				\begin{pmatrix}
					0 & 0 & M^2_z\\
					0 & 0 & 0\\
					0 & 0 & 0
				\end{pmatrix}\\
				&=
				\begin{pmatrix}
					0 & M^2_z & 0\\
					0 & 0 & 0\\
					0 & 0 & 0
				\end{pmatrix}\\&
				= \begin{pmatrix}
					I_{H^2} & 0 & 0\\
					0 & I_{H^2} & 0\\
					0 & 0 & 0
				\end{pmatrix}
				\begin{pmatrix}
					0 & M^2_z & 0\\
					0 & 0 & 0\\
					0 & 0 & 0
				\end{pmatrix}
				\begin{pmatrix}
					I_{H^2} & 0 & 0\\
					0 & I_{H^2} & 0\\
					0 & 0 & 0
				\end{pmatrix}\\
				&=
				D_{S_3}\tilde{G}_2D_{S_3},
			\end{aligned}
		\end{equation*}
		
		\begin{equation*}
			\begin{aligned}
				S_2 - \tilde{S}^*_1S_3 &=
				\begin{pmatrix}
					0 & M_z & M_z\\
					0 & 0 & M_z\\
					0 & 0 & 0
				\end{pmatrix} -
				\begin{pmatrix}
					M^*_z & 0 & 0\\
					M^*_z & M^*_z & 0\\
					0 & M^*_z & M^*_z
				\end{pmatrix}
				\begin{pmatrix}
					0 & 0 & M^2_z\\
					0 & 0 & 0\\
					0 & 0 & 0
				\end{pmatrix}\\
				&=
				\begin{pmatrix}
					0 & M_z & 0\\
					0 & 0 & 0\\
					0 & 0 & 0
				\end{pmatrix}
				\\&= \begin{pmatrix}
					I_{H^2} & 0 & 0\\
					0 & I_{H^2} & 0\\
					0 & 0 & 0
				\end{pmatrix}
				\begin{pmatrix}
					0 & M_z & 0\\
					0 & 0 & 0\\
					0 & 0 & 0
				\end{pmatrix}
				\begin{pmatrix}
					I_{H^2} & 0 & 0\\
					0 & I_{H^2} & 0\\
					0 & 0 & 0
				\end{pmatrix}\\
				&=
				D_{S_3}2G_2D_{S_3},
			\end{aligned}
		\end{equation*} $${\rm{and}}$$
		\begin{equation*}
			\begin{aligned}
				\tilde{S}_1 - S^*_2S_3 &=
				\begin{pmatrix}
					M_z & M_z & 0\\
					0 & M_z & M_z\\
					0 & 0 & M_z
				\end{pmatrix} -
				\begin{pmatrix}
					0 & 0 & 0\\
					M^*_z & 0 & 0\\
					M^*_z & M^*_z & 0
				\end{pmatrix}
				\begin{pmatrix}
					0 & 0 & M^2_z\\
					0 & 0 & 0\\
					0 & 0 & 0
				\end{pmatrix}\\
				&=
				\begin{pmatrix}
					M_z & M_z & 0\\
					0 & M_z & 0\\
					0 & 0 & 0
				\end{pmatrix}
				\\&= \begin{pmatrix}
					I_{H^2} & 0 & 0\\
					0 & I_{H^2} & 0\\
					0 & 0 & 0
				\end{pmatrix}
				\begin{pmatrix}
					M_z & M_z & 0\\
					0 & M_z & 0\\
					0 & 0 & 0
				\end{pmatrix}
				\begin{pmatrix}
					I_{H^2} & 0 & 0\\
					0 & I_{H^2} & 0\\
					0 & 0 & 0
				\end{pmatrix}\\
				&=
				D_{S_3}2\tilde{G}_1D_{S_3}.
			\end{aligned}
		\end{equation*}
One can easily verify that  $G_1, G_2, \tilde{G}_1, \tilde{G}_2$ commute with each other. We notice that
		\begin{equation}\label{g11}
			\begin{aligned}
				[G^*_1, G_1] =
				\begin{pmatrix}
					- I_{H^2} & 0\\
					0 & I_{H^2}
				\end{pmatrix} &\ne
				\begin{pmatrix}
					- M^2_zM^{*2}_z & 0\\
					0 & I_{H^2}
				\end{pmatrix} = [\tilde{G}^*_2, \tilde{G}_2],
			\end{aligned}
		\end{equation}
		$${\rm{and}}$$
		\begin{equation}\label{g12}
			\begin{aligned}
				[2G^*_2, 2G_2] =
				\begin{pmatrix}
					I_{H^2} & I_{H^2}\\
					I_{H^2} & 2I_{H^2}
				\end{pmatrix} &\ne
				\begin{pmatrix}
					I_{H^2} - 2M_zM^*_z & I_{H^2} - M_zM^*_z\\
					I_{H^2} - M_zM^*_z & 2I_{H^2} - M_zM^*_z
				\end{pmatrix} = [2\tilde{G}^*_1, 2\tilde{G}_1].
			\end{aligned}
		\end{equation}
Hence from \eqref{g11} and \eqref{g12}, we see that the conditions in $(2)(d)$ and $(2)(e)$ in Proposition \ref{Prop Partial 2} are not satisfied. Thus, we conclude that  the set of sufficient conditions for the existence of $\Gamma_{E(3; 2; 1, 2)}$-isometric dilation presented in Theorem \ref{condilation1} are not necessary in general.
	\end{ex}
	
	\section{Families of $\Gamma_{E(3; 3; 1, 1, 1)}$-Contractions and $\Gamma_{E(3; 2; 1, 2)}$-Contractions and Their Dilations}\label{Sec 5}
In this section, we construct explicit $\Gamma_{E(3; 3; 1, 1, 1)} $-isometric and $\Gamma_{E(3; 2; 1; 2)} $-isometric dilations of $\Gamma_{E(3; 3; 1, 1, 1)}$-contraction and $\Gamma_{E(3; 2; 1; 2)}$-contraction, respectively. Let $\mathcal E$ be a Hilbert space, and $\ell_2(\mathcal E)$ denotes the Hilbert space of infinite direct sums $\mathcal E\oplus \mathcal {E}\cdots. $ Let $H^{2}_{\mathbb D}(\mathcal E)$ denote the Hardy space of $\mathcal E$-valued functions defined on $\mathbb D.$	
%\begin{thm}
		%Let $\textbf{V} = (V_1, V_2, V_3)$ be a commuting triple of bounded operators on a Hilbert space $\mathcal{H}$. Then the following are equivalent:
		%\begin{enumerate}
			%\item $\textbf{V}$ is a $\mathbb{E}$-isometry.
			
			%\item $\textbf{V}$ is a $\mathbb{E}$-contraction and $V_3$ is an isometry.
			
			%\item $V_1 = V^*_2V_3$, $V_2$ is a contraction and $V_3$ is an isometry.
			
			%\item $V_1 = V^*_2V_3$, the spectral radii $r(V_1), r(V_2)$ are not bigger than one, and $V_3$ is an isometry.
		%\end{enumerate}
	%\end{thm}
	
	%We are now ready to proceed for constructing example of $\Gamma_4$-contraction.
\begin{ex}\label{Exam 3}
		Let us consider the Hilbert space $\mathcal{H} = \underbrace{\ell_2(\mathbb{C}^2) \oplus \dots \oplus \ell_2(\mathbb{C}^2)}_{4 \,\, \text{times}}$. Let $A_{\alpha},B,P$ be the operators on $\mathcal H$ of the following form
		\begin{equation*}
			\begin{aligned}
				A_{\alpha} &=
				\begin{pmatrix}
					G & 0 & 0 & 0\\
					0 & 0 & 0 & 0\\
					0 & 0 & 0 & 0\\
					0 & 0 & 0 & 0
				\end{pmatrix},
				B =
				\begin{pmatrix}
					0 & 0 & 0 & 0\\
					0 & 0 & 0 & 0\\
					0 & 0 & 0 & 0\\
					0 & 0 & 0 & 0
				\end{pmatrix} \,\, \text{and} \,\,
				P =
				\begin{pmatrix}
					0 & 0 & 0 & 0\\
					0 & 0 & M_z & 0\\
					0 & - M_z & 0 & 0\\
					0 & 0 & 0 & 0
				\end{pmatrix},
			\end{aligned}
		\end{equation*}
		where $M_z$ denotes  the unilateral shift of multiplicity equal to the dimension of $\mathcal E$ and $G$ on $\ell_2(\mathbb C^2)$ is defined by
$$G(c_0,c_1,\ldots):=(G_1c_0,0,\cdots)~{\rm{for}}~(c_0,c_1,\ldots)\in \ell_2(\mathbb C^2)$$ and $G_1$ is of the form $G_1=\begin{pmatrix} 0 & \alpha\\ 0 & 0\end{pmatrix}$ for all $\alpha \in \mathbb {\bar{D}}.$ Let us set  $$\textbf{T} =(T_1,T_2,\dots,T_7)= (A_{\alpha}, A_{\alpha}, B,A_{\alpha},  B, B, P).$$  It is noted that
  \begin{equation*}
   \begin{aligned}
    D^2_P &=
    \begin{pmatrix}
     I_{\ell_2(\mathbb{C}^2)} & 0 & 0 & 0\\
     0 & 0 & 0 & 0\\
     0 & 0 & 0 & 0\\
     0 & 0 & 0 & I_{\ell_2(\mathbb{C}^2)}
    \end{pmatrix} = D_P.
   \end{aligned}
  \end{equation*}
The defect space of $P$ is given by $\mathcal{D}_P = \ell_2(\mathbb{C}^2) \oplus \{0\} \oplus \{0\} \oplus \ell_2(\mathbb{C}^2)$. To proceed, define
  \begin{equation}\label{f11}
   \begin{aligned}
(F_1,F_2,F_3,F_4,F_5,F_6)=
    \left(\left(\begin{smallmatrix}
     G & 0 & 0 & 0\\
     0 & 0 & 0 & 0\\
     0 & 0 & 0 & 0\\
     0 & 0 & 0 & 0
    \end{smallmatrix}\right),\left(\begin{smallmatrix}
     G & 0 & 0 & 0\\
     0 & 0 & 0 & 0\\
     0 & 0 & 0 & 0\\
     0 & 0 & 0 & 0
    \end{smallmatrix}\right),\left(\begin{smallmatrix}
     0 & 0 & 0 & 0\\
     0 & 0 & 0 & 0\\
     0 & 0 & 0 & 0\\
     0 & 0 & 0 & 0
    \end{smallmatrix} \right),\left(\begin{smallmatrix}
     G & 0 & 0 & 0\\
     0 & 0 & 0 & 0\\
     0 & 0 & 0 & 0\\
     0 & 0 & 0 & 0
    \end{smallmatrix}\right),\left(\begin{smallmatrix}
     0 & 0 & 0 & 0\\
     0 & 0 & 0 & 0\\
     0 & 0 & 0 & 0\\
     0 & 0 & 0 & 0
    \end{smallmatrix}\right),\left(\begin{smallmatrix}
     0 & 0 & 0 & 0\\
     0 & 0 & 0 & 0\\
     0 & 0 & 0 & 0\\
     0 & 0 & 0 & 0
    \end{smallmatrix}\right)
\right)\end{aligned}
  \end{equation}
on $\mathcal{D}_P$. With these definitions in place, we observe that
  \begin{equation}\label{Fundamental (A, B, P)}
   \begin{aligned}
    A_{\alpha} - B^*P &= D_P\left(\begin{smallmatrix}
     G & 0 & 0 & 0\\
     0 & 0 & 0 & 0\\
     0 & 0 & 0 & 0\\
     0 & 0 & 0 & 0
    \end{smallmatrix}\right)D_P, \,\, ~~~B - A^*_{\alpha}P = 0.\end{aligned}
  \end{equation}

From \eqref{f11} and \eqref{Fundamental (A, B, P)}, it then follows that
\begin{equation}
\begin{aligned}
T_i-T_{7-i}^*T_7=D_P\left(\begin{smallmatrix}
     G & 0 & 0 & 0\\
     0 & 0 & 0 & 0\\
     0 & 0 & 0 & 0\\
     0 & 0 & 0 & 0
    \end{smallmatrix}\right)D_P, &1\leq i \leq 2,~T_3-T_4^*T_7=0, T_4-T_3^*T_7=D_P\left(\begin{smallmatrix}
     G & 0 & 0 & 0\\
     0 & 0 & 0 & 0\\
     0 & 0 & 0 & 0\\
     0 & 0 & 0 & 0
    \end{smallmatrix}\right)D_P
\end{aligned}
\end{equation}
$${\rm{and}}$$
\begin{equation}
T_i-T_{7-i}^*T_7=0, 5 \leq i \leq 6.
\end{equation}
Because $F_j = 0$ for $j=3,5,6$, it follows that $[F_j, F^*_j] = 0$. On the other hand, since $G$ is not normal, a straightforward computation shows that
  \begin{equation*}
   \begin{aligned}
    [F_i, F^*_i] &=
    \begin{pmatrix}
     GG^* - G^*G & 0 & 0 & 0\\
     0 & 0 & 0 & 0\\
     0 & 0 & 0 & 0\\
     0 & 0 & 0 & 0
    \end{pmatrix} \ne 0
   \end{aligned}
  \end{equation*}
for $i=1,2,4.$ Thus, we conclude that $[F_i, F^*_i] \ne [F_{7-i}, F^*_{7-i}]$ for $1 \leqslant i \leqslant 6$. Furthermore, as $G^2_1 = 0$, it implies that $G^2 = 0$. As a result, we have $F_i^2 = 0$ for $i=1,2,4.$  
In this context, we do not provide a direct proof that $\textbf{T}$ is a $\Gamma_{E(3; 3; 1, 1, 1)}$-contraction. Instead, we will consider the dilation of a $\Gamma_{E(3; 3; 1, 1, 1)}$-contraction $\textbf{T}$ on a larger Hilbert space, which, by definition, indicates that it is a $\Gamma_{E(3; 3; 1, 1, 1)}$-contraction. 
		
\noindent{\textbf{Construction of $\Gamma_{E(3; 3; 1, 1, 1)}$-Isometric Dilation:}} Let $F_i=F$ for $i=1,2,4$ and
		\[\mathcal{K} = \mathcal{H} \oplus \mathcal{D}_P \oplus \mathcal{D}_P \oplus \cdots.\]
We define the following operators  on $\mathcal{K}$ by
		\begin{equation}\label{V_i, V_j}
			\begin{aligned}
				W_1&=V_i &=
				\begin{pmatrix}
					A_{\alpha} & 0 & 0 & 0 & \dots\\
					F^*D_P & F & 0 & 0 &  \dots\\
					0 & F^* & F & 0 & \dots\\
					0 & 0 & F^* & F &  \dots\\
					0 & 0 & 0 & F^* & \dots\\
					\vdots & \vdots & \vdots & \vdots & \ddots 
				\end{pmatrix}\,\,~{\rm{for}}~i=1,2,4,
				W_2&=V_j =
				\begin{pmatrix}
					B & 0 & 0 & 0 & \dots\\
					FD_P & 0 & 0 & 0 &  \dots\\
					F^*D_P & F & 0 & 0 & \dots\\
					0 & F^* & F & 0 & \dots\\
					0 & 0 & F^* & F & \dots\\
					\vdots & \vdots & \vdots & \vdots & \ddots 
				\end{pmatrix}~{\rm{for}}~j=3,5,6
			\end{aligned}
		\end{equation}
		 $${\rm{and}}$$
		\begin{equation}\label{V_7}
			\begin{aligned}
				V_7 =
				\begin{pmatrix}
					P & 0 & 0 & 0 & \dots\\
					0 & 0 & 0 & 0 &  \dots\\
					D_P & 0 & 0 & 0 & \dots\\
					0 & I & 0 & 0 & \dots\\
					0 & 0 & I & 0 & \dots\\
					\vdots & \vdots & \vdots & \vdots & \ddots 
				\end{pmatrix}.
			\end{aligned}
		\end{equation}
		We prove that $\textbf{V} = (V_1, \dots, V_7)$ is a $\Gamma_{E(3; 3; 1, 1, 1)}$-isometry. According to [Theorem $4.4$, \cite{apal2}] , we need to verify the following:
		\begin{enumerate}
			\item $V_1, \dots, V_7$ commute with each other,
			
			\item $V_i = V^*_{7-i}V_7, r(V_i) \leqslant 1$ for $1 \leqslant i \leqslant 6$,
			\item $V_7$ isometry.
		\end{enumerate}
Clearly,  $V_7$ is an isometry. 
		
\noindent{\textbf{Step 1:}} First we show that $V_iV_j=V_jV_i$ for $1\leq i,j \leq 7.$  If we can show that $W_1W_2=W_2W_1$ and $ W_iV_7=V_7W_i,1\leq i \leq 2,$ then we are done. Observe that 
		\begin{equation}\label{V_iV_j}
			\begin{aligned}
				W_1W_2&=
				\begin{pmatrix}
					A_{\alpha}B & 0 & 0 & 0 & \dots\\
					F^*D_PB + F^2D_P & 0 & 0 & 0 &  \dots\\
					F^*FD_P + FF^*D_P & F^2 & 0 & 0 & \dots\\
					F^*D_P & F^*F + FF^* & F^2 & 0 & \dots\\
					0 & F^{*2} & F^*F + FF^* & F^2 & \dots\\
					\vdots & \vdots & \vdots & \vdots & \ddots 
				\end{pmatrix}
			\end{aligned}
		\end{equation}
		and
		\begin{equation}\label{V_jV_i}
			\begin{aligned}
				W_2W_1&=
				\begin{pmatrix}
					BA_{\alpha} & 0 & 0 & 0 & \dots\\
					FD_PA_{\alpha} & 0 & 0 & 0 &  \dots\\
					F^*D_PA_{\alpha} + FF^*D_P & F^2 & 0 & 0 & \dots\\
					F^*D_P & F^*F + FF^* & F^2 & 0 &  \dots\\
					0 & F^{*2} & F^*F + FF^* & F^2 &\dots\\
					\vdots & \vdots & \vdots & \vdots & \ddots 
				\end{pmatrix}.
			\end{aligned}
		\end{equation}
We first show that $(2, 1)$ entries of 	$W_1W_2$ and $W_2W_1$  are same. To show this, we need to prove $F^*D_PB + F^2D_P=FD_PA_{\alpha}.$ As $F^2=B=0$, one can easily show that $F^*D_PB + F^2D_P=FD_PA_{\alpha}.$ Note that
\begin{equation}\label{V_i, V_j commute}
			\begin{aligned}
				F^*FD_P + FF^*D_P
				&=
				\begin{pmatrix}
					G^*G + GG^* & 0 & 0 & 0\\
					0 & 0 & 0 & 0\\
					0 & 0 & 0 & 0\\
					0 & 0 & 0 & 0
				\end{pmatrix} =
				F^*D_PA_{\alpha} + FF^*D_P.
			\end{aligned}
		\end{equation}
This implies that $(3, 1)$ entries of $W_1W_2$ and $W_2W_1$  are identical.  This shows that $W_1W_2=W_2W_1.$ We now show that $ W_iV_7=V_7W_i,1\leq i \leq 2.$ Notice that \begin{equation}\label{V_iV_7, V_7V_i}
			\begin{aligned}
				W_1V_7 &=
				\begin{pmatrix}
					A_{\alpha}P & 0 & 0 & 0 & \dots\\
					F^*D_PP & 0 & 0 & 0 & \dots\\
					FD_P & 0 & 0 & 0 & \dots\\
					F^*D_P & F & 0 & 0 & \dots\\
					0 & F^* & F & 0 & \dots\\
					0 & 0 & F^* & F & \dots\\
					\vdots & \vdots & \vdots & \vdots & \ddots 
				\end{pmatrix}
			\end{aligned}
			\text{and} \,\,\,\,
			V_7W_1 =
			\begin{pmatrix}
				PA_{\alpha} & 0 & 0 & 0 & \dots\\
				0 & 0 & 0 & 0 & \dots\\
				D_PA_{\alpha} & 0 & 0 & 0 & \dots\\
				F^*D_P & F & 0 & 0 & \dots\\
				0 & F^* & F & 0 & \dots\\
				0 & 0 & F^* & F & \dots\\
				\vdots & \vdots & \vdots & \vdots & \ddots 
			\end{pmatrix}.
		\end{equation}
		%Note that $F^*D_PP = 0$
		%\begin{equation*}
			%\begin{aligned}
				%F^*_1D_PP &=
				%\begin{pmatrix}
					%R^* & 0 & 0 & 0\\
					%0 & 0 & 0 & 0\\
					%0 & 0 & 0 & 0\\
					%0 & 0 & 0 & 0
				%\end{pmatrix}
				%\begin{pmatrix}
					%I_{l^2(\mathbb{C}^2)} & 0 & 0 & 0\\
					%0 & 0 & 0 & 0\\
					%0 & 0 & 0 & 0\\
					%0 & 0 & 0 & I_{l^2(\mathbb{C}^2)}
				%\end{pmatrix}
				%\begin{pmatrix}
					%0 & 0 & 0 & 0\\
					%0 & 0 & M_z & 0\\
					%0 & - M_z & 0 & 0\\
					%0 & 0 & 0 & 0
				%\end{pmatrix}
				%=
				%\begin{pmatrix}
					%0 & 0 & 0 & 0\\
					%0 & 0 & 0 & 0\\
					%0 & 0 & 0 & 0\\
					%0 & 0 & 0 & 0
				%\end{pmatrix}
			%\end{aligned}
		%\end{equation*}
It yields from \eqref{V_iV_7, V_7V_i} that each entry of the operator matrix $W_1V_7$ is identical to the corresponding entry of the the operator matrix $V_7W_1$. This demonstrates that $W_1V_7=V_7W_1.$ Similarly, we also observe that \begin{equation}\label{V_jV_7, V_7V_j}
			\begin{aligned}
				W_2V_7 &=
				\begin{pmatrix}
					BP & 0 & 0 & \dots\\
					FD_PP & 0 & 0 & \dots\\
					F^*D_PP & 0 & 0 & \dots\\
					FD_P & 0 & 0 & \dots\\
					F^*D_P & F & 0 & \dots\\
					0 & F^* & F & \dots\\
					\vdots & \vdots & \vdots & \ddots 
				\end{pmatrix}
			\end{aligned}
			\text{and} \,\,\,\,
			V_7W_2=
			\begin{pmatrix}
				PB & 0 & 0 & \dots\\
				0 & 0 & 0 & \dots\\
				D_PB & 0 & 0 & \dots\\
				FD_P & 0 & 0 & \dots\\
				F^*D_P & F & 0 & \dots\\
				0 & F^* & F & \dots\\
				\vdots & \vdots & \vdots & \ddots 
			\end{pmatrix}.
		\end{equation}
From \eqref{V_iV_7, V_7V_i}, it follows that each entry of the operator matrix $W_2V_7$ is identical to the corresponding entry of the the operator matrix $V_7W_2$. This implies that $W_2V_7=V_7W_2.$		
		
\noindent{\textbf{Step 2:}} In order to demonstrate that  $V_i = V^*_{7-i}V_7$ for $1 \leqslant i \leqslant 6,$ we observe that
		\begin{equation}\label{V_i = V^*_{7-i}V_7}
			\begin{aligned}
				V^*_{7-i}V_7 &=
				\begin{pmatrix}
					B^* & D_PF^* & D_PF & 0 & 0 & \dots\\
					0 & 0 & F^* & F & 0 &  \dots\\
					0 & 0 & 0 & F^* & F & \dots\\
					0 & 0 & 0 & 0 & F^* & \dots\\
					0 & 0 & 0 & 0 & 0 & \dots\\
					\vdots & \vdots & \vdots & \vdots & \vdots & \ddots 
				\end{pmatrix}
				\begin{pmatrix}
					P & 0 & 0 & 0 & \dots\\
					0 & 0 & 0 & 0 &  \dots\\
					D_P & 0 & 0 & 0 & \dots\\
					0 & I & 0 & 0 & \dots\\
					0 & 0 & I & 0 & \dots\\
					\vdots & \vdots & \vdots & \vdots & \ddots 
				\end{pmatrix}\\
				&=
				\begin{pmatrix}
					B^*P + D_PFD_P& 0 & 0 & 0 & 0 &\dots\\
					F^*D_P & F & 0 & 0 & 0 & \dots\\
					0 & F^* & F & 0 & 0 & \dots\\
					0 & 0 & F^* & F & 0 & \dots\\
					0 & 0 & 0 & F^* & F & \dots\\
					\vdots & \vdots & \vdots & \vdots & \vdots & \ddots 
				\end{pmatrix}.
			\end{aligned}
		\end{equation}
As $B^*P + D_PFD_P=A_{\alpha}$, 	we can derive from \eqref{V_i = V^*_{7-i}V_7} that
		\begin{equation*}
			\begin{aligned}
				V^*_{7-i}V_7 &=
				\begin{pmatrix}
					A_{\alpha} & 0 & 0 & 0 & \dots\\
					F^*D_P & F & 0 & 0 &  \dots\\
					0 & F^* & F & 0 & \dots\\
					0 & 0 & F^* & F &  \dots\\
					0 & 0 & 0 & F^* & \dots\\
					\vdots & \vdots & \vdots & \vdots & \ddots 
				\end{pmatrix}\\& = V_i.
			\end{aligned}
		\end{equation*}
				
\noindent{\textbf{Step 3:}} We now calculate norm $\|W_i\|$ for $1\leq i \leq 2.$  Note that
		\begin{equation}\label{V^*_iV_i}
			\begin{aligned}
				W^*_1W_1&=
				\begin{pmatrix}
					A^*_{\alpha}A_{\alpha} + D_PFF^*D_P & 0 & 0 & 0 & \dots\\
					0 & F^*F + FF^* & 0 & 0 &  \dots\\
					0 & 0 & F^*F + FF^* & 0 & \dots\\
					0 & 0 & 0 & F^*F + FF^* &  \dots\\
					\vdots & \vdots & \vdots & \vdots & \ddots 
				\end{pmatrix}.
			\end{aligned}
		\end{equation}
Since $F = A_{\alpha}$ and $D_PFF^*D_P = FF^*,$ it implies from \eqref{V^*_iV_i} that
		\begin{equation}\label{w11}
			\begin{aligned}
				||W_1|^2
				&= ||W^*_1W_1||\\
				&= ||F^*F + FF^*||\\
				&= ||G^*G + GG^*||\\
				&= 
				\left|\left|
				\begin{pmatrix}
					|\alpha|^2 & 0\\
					0 & |\alpha|^2
				\end{pmatrix}\right|\right|\\
				&= |\alpha|^2\\
				&\leqslant 1.
			\end{aligned}
		\end{equation}
Because $V_7$ is isometry and $W_2=W_1^*V_7$, we deduce from \eqref{w11} that $\|W_2\|\leq 1.$ Therefore, by [Theorem $4.4$, \cite{apal2}], we conclde that $\textbf{V} = (V_1, \dots, V_7)$ is a $\Gamma_{E(3; 3; 1, 1, 1)}$-isometry. 
	\end{ex}
\begin{rem}
In Example \ref{Exam 3}, we have seen that $\textbf{T} = (A_{\alpha}, A_{\alpha}, B,A_{\alpha},  B, B, P)$ has an $\Gamma_{E(3; 3; 1, 1, 1)} $-isometric dilation $\textbf{V} = (V_1, \dots, V_7)$. Since $\textbf{V} = (V_1, \dots, V_7)$ is a $\Gamma_{E(3; 3; 1, 1, 1)}$-isometry, it follows from [Theorem $4.8$, \cite{apal2}] that $\textbf{W} = (V_1, V_3 + V_5, V_7, V_2 + V_4, V_6)$ is a $\Gamma_{E(3; 2; 1, 2)}$-isometry. It implies from [Theorem $4.5$, \cite{apal2}] that $\textbf{W}$ is a $\Gamma_{E(3; 2; 1, 2)} $-contraction and so being the restriction of the invariant subspace $\mathcal H$, $(A_{\alpha}, A_{\alpha} + B, P, A_{\alpha} + B, B) $ is a  $\Gamma_{E(3; 2; 1, 2)}$-contraction for all $\alpha \in \overline{\mathbb D}. $ Still now we have not identified an example of $\Gamma_{E(3; 3; 1, 1, 1)}$-contraction (respectively, $\Gamma_{E(3; 2; 1, 2)}$-contraction), which fails to satisfy one of the necessary conditions outlined in Theorem \ref{Necessary Conditions 1} (respectively, Theorem \ref{Necessary Conditions 2}). Thus, the existence of $\Gamma_{E(3; 3; 1, 1, 1)}$-isometric dilation (respectively, $\Gamma_{E(3; 2; 1, 2)}$-isometric dilation) is still open.\end{rem}	
	
\section{A Family of $\mathcal{\bar{P}}$-Contraction and Their Isometric Dilation}\label{Sec 6}
The existence of $\mathcal{\bar{P}}$-isometric dilation for a $\mathcal{\bar{P}}$-contraction is still unknown. However, we construct a family of $\mathcal{\bar{P}}$-contractions that have $\mathcal{\bar{P}}$-isometric dilation in this section.
\begin{lem}\label{g1112}
Let $G$ be defined as in Example \ref{Exam 3}. Then the following statements hold:
\begin{enumerate}
\item $G^*(I_{l^2(\mathbb{C}^2)} - \frac{1}{4}(G^*G + GG^*)) =(I_{l^2(\mathbb{C}^2)} - \frac{1}{4}(G^*G + GG^*)) G^*.$
\item $G(I_{l^2(\mathbb{C}^2)} - \frac{1}{4}(G^*G + GG^*)) =(I_{l^2(\mathbb{C}^2)} - \frac{1}{4}(G^*G + GG^*)) G.$
\end{enumerate}
\end{lem}
\begin{proof}	
Note that
\begin{equation}\label{G111}
\begin{aligned}
\|G^*G + GG^*\|&=\|G^*_1G_1 + G_1G^*_1\|\\& =
	\left|\left|
	\begin{pmatrix}
		|\alpha|^2 & 0\\
		0 & |\alpha|^2
	\end{pmatrix}\right|\right|\\&\leq 1.
	\end{aligned}
	\end{equation}
It follows from \eqref{G111} that $\frac{\|G^*G + GG^*\|}{4}\leq 1,$ which is equivalent to the condition $$I_{l^2(\mathbb{C}^2)} - \frac{1}{4}(G^*G + GG^*)>0.$$  Observe that
\begin{equation}\label{F 2}
\begin{aligned}
G(I_{l^2(\mathbb{C}^2)} - \frac{1}{4}(G^*G + GG^*))
&= G - \frac{1}{4} (GG^{*}G + G^2G^*)\\
&= G - \frac{1}{4} GG^*G\\
&= G - \frac{1}{4} (GG^*G + G^*G^{2})\\ 
&= (I_{l^2(\mathbb{C}^2)} - \frac{1}{4}(G^*G + GG^*)) G.
\end{aligned}
\end{equation}
It implies from [Page $153$, \cite{nyoung}] that
\begin{equation}\label{F 3}
\begin{aligned}
G(I_{l^2(\mathbb{C}^2)} - \frac{1}{4}(G^*G + GG^*))^{1/2} &= (I_{l^2(\mathbb{C}^2)} - \frac{1}{4}(G^*G + GG^*))^{1/2}G.
\end{aligned}
\end{equation}
Using a similar argument, we can also demonstrate that
	\begin{equation}\label{F 4}
		\begin{aligned}
			G^*(I_{l^2(\mathbb{C}^2)} - \frac{1}{4}(G^*G + GG^*))^{1/2} &= (I_{l^2(\mathbb{C}^2)} - \frac{1}{4}(G^*G + GG^*))^{1/2}G^*.
		\end{aligned}
	\end{equation}
This completes the proof.
\end{proof}	
	\begin{ex}\label{Exam 5}
		Let $\mathcal{H} = \underbrace{\ell_2(\mathbb{C}^2) \oplus \dots \oplus \ell_2(\mathbb{C}^2)}_{4 \,\, \text{times}}.$  Let  $A_{\alpha}, S_{\alpha}, P$ be the operators on $\mathcal H$ of the following form:
		\begin{equation*}
			\begin{aligned}
				A_{\alpha} &=
				\begin{pmatrix}
					(I_{\ell_2(\mathbb{C}^2)} - \frac{1}{4}(G^*G + GG^*))^{1/2} & 0 & 0 & 0\\
					0 & I_{\ell_2(\mathbb{C}^2)} & 0 & 0\\
					0 & 0 & I_{\ell_2(\mathbb{C}^2)} & 0\\
					0 & 0 & 0 & I_{\ell_2(\mathbb{C}^2)}
				\end{pmatrix}, \hspace{0.5cm}
				S_{\alpha} =
				\begin{pmatrix}
					G & 0 & 0 & 0\\
					0 & 0 & 0 & 0\\
					0 & 0 & 0 & 0\\
					0 & 0 & 0 & 0
				\end{pmatrix},
			\end{aligned}
		\end{equation*}
		$${\rm{and}}$$
		\begin{equation*}
			\begin{aligned}
				P &=
				\begin{pmatrix}
					0 & 0 & 0 & 0\\
					0 & M_z & 0 & 0\\
					0 & 0 & 0 & M_z\\
					0 & 0 & -M_z & 0
				\end{pmatrix},
			\end{aligned}
		\end{equation*}
		where $G$ is defined as in Example \ref{Exam 3}. Clearly, $S_{\alpha}P=PS_{\alpha}. $ By Lemma \ref{g1112}, we have $A_{\alpha}S_{\alpha}=S_{\alpha}A_{\alpha}. $ Furthermore, a simple calculation demonstrates that $A_{\alpha}P=PA_{\alpha}. $ Therefore, we conclude that $(A_{\alpha}, S_{\alpha}, P)$ is a commuting triple of bounded operators on $\mathcal H$. Notice that		\begin{equation}
			\begin{aligned}
				D^2_P = I - P^*P &=
				\begin{pmatrix}
					I_{\ell_2(\mathbb{C}^2)} & 0 & 0 & 0\\
					0 & I_{\ell_2(\mathbb{C}^2)} & 0 & 0\\
					0 & 0 & I_{\ell_2(\mathbb{C}^2)} & 0\\
					0 & 0 & 0 & I_{\ell_2(\mathbb{C}^2)}
				\end{pmatrix} -
				\begin{pmatrix}
					0 & 0 & 0 & 0\\
					0 & M_z^* & 0 & 0\\
					0 & 0 & 0 & -M_z^*\\
					0 & 0 & M_z^* & 0
				\end{pmatrix}
				\begin{pmatrix}
					0 & 0 & 0 & 0\\
					0 & M_z & 0 & 0\\
					0 & 0 & 0 & M_z\\
					0 & 0 & -M_z & 0
				\end{pmatrix}\\
				&=
				\begin{pmatrix}
					I_{\ell_2(\mathbb{C}^2)} & 0 & 0 & 0\\
					0 & 0 & 0 & 0\\
					0 & 0 & 0 & 0\\
					0 & 0 & 0 & 0
				\end{pmatrix}\\
				&= D_P.
			\end{aligned}
		\end{equation}
Thus, the defect space of $P$ is $\mathcal{D}_P = l^2(\mathbb{C}^2) \oplus \{0\} \oplus \{0\} \oplus \{0\}$. Let us set
		\begin{equation*}
			\begin{aligned}
				F &=
				\begin{pmatrix}
					G & 0 & 0 & 0\\
					0 & 0 & 0 & 0\\
					0 & 0 & 0 & 0\\
					0 & 0 & 0 & 0
				\end{pmatrix} = S_{\alpha}.
			\end{aligned}
		\end{equation*}
Observe that
		\begin{equation}\label{S alpha, P Fundamental}
			\begin{aligned}
				S_{\alpha} - S^*_{\alpha}P
				&= S_{\alpha} = F = D_PFD_P.
			\end{aligned}
		\end{equation}
		
		\noindent\textbf{Construction of $\mathcal{\bar{P}}$-isometric dilation:} Let $\mathcal{K} = \mathcal{H} \oplus \mathcal{D}_P \oplus \mathcal{D}_P \oplus \dots. $ We consider a triple of bounded operators $(R_1, R_2, R_3)$ of the following form:
		\begin{equation}\label{R_1, R_2}
			\begin{aligned}
				R_1 &=
				\begin{pmatrix}
					A_{\alpha} & 0 & 0 & 0 & \dots\\
					0 & L & 0 & 0 & \dots\\
					0 & 0 & L & 0 & \dots\\
					0 & 0 & 0 & L & \dots\\
					\vdots & \vdots & \vdots & \vdots & \ddots
				\end{pmatrix}, \,\,
				R_2 =
				\begin{pmatrix}
					S_{\alpha} & 0 & 0 & 0 & \dots\\
					F^*D_P & F & 0 & 0 &\dots\\
					0 & F^* & F & 0 & \dots\\
					0 & 0 & F^* & F & \dots\\
					\vdots & \vdots & \vdots & \vdots & \ddots
				\end{pmatrix}
			\end{aligned}
		\end{equation}
		$${\rm{and}}$$
		\begin{equation}\label{R_3}
			\begin{aligned}
				R_3 &=
				\begin{pmatrix}
					P & 0 & 0 & 0 & \dots\\
					D_P & 0 & 0 & 0 & \dots\\
					0 & I & 0 & 0 & \dots\\
					0 & 0 & I & 0 & \dots\\
					\vdots & \vdots & \vdots & \vdots & \ddots
				\end{pmatrix},
			\end{aligned}
		\end{equation}
		where $L = (I_{\mathcal{H}} - \frac{1}{4}(F^*F + FF^*))^{1/2}$.
		
In order to show that $(R_1, R_2, R_3)$ is a $\mathcal{\bar{P}}$-isometry,  we must verify the following properties as described in [Theorem $5.2$, \cite{Jindal}]:
\begin{enumerate}
\item $(R_1, R_2, R_3)$ is a commuting triple,
\item $(R_2, R_3)$ is a $\Gamma$-isometry,
\item $R^*_1R_1 = I - \frac{1}{4}R^*_2R_2$.
\end{enumerate}
		
\noindent{\textbf{Step 1:}} We now prove that $(R_1, R_2, R_3)$ is a commuting triple. Note that
		\begin{equation}\label{R_1R_2}
			\begin{aligned}
				R_1R_2 &=
				\begin{pmatrix}
					A_{\alpha}S_{\alpha} & 0 & 0 & 0 & \dots\\
					LF^*D_P & LF & 0 & 0 & \dots\\
					0 & LF^* & LF & 0 & \dots\\
					0 & 0 & LF^* & LF & \dots\\
					\vdots & \vdots & \vdots & \vdots & \ddots
				\end{pmatrix}
			\end{aligned}
		\end{equation}
		$${\rm{and}}$$
		\begin{equation}\label{R_2R_1}
			\begin{aligned}
				R_2R_1 =
				\begin{pmatrix}
					S_{\alpha}A_{\alpha} & 0 & 0 & 0 & \dots\\
					F^*D_PA_{\alpha} & FL & 0 & 0 & \dots\\
					0 & F^*L & FL & 0 & \dots\\
					0 & 0 & F^*L & FL & \dots\\
					\vdots & \vdots & \vdots & \vdots & \ddots
				\end{pmatrix}.
			\end{aligned}
		\end{equation}
Note that $LF^*D_P = LS^*_{\alpha}$ and $F^*D_PA_{\alpha} = S^*_{\alpha}A_{\alpha}$. It follows from  $(\ref{F 3})$  and Lemma \ref{g1112} that $LS^*_{\alpha} = S^*_{\alpha}A_{\alpha}$. Thus, we deduce that the $(2,1)$ entries of $R_1R_2$ and $R_2R_1$ are same. It yields from  Lemma \ref{g1112} that $LF = FL$ and $LF^* = F^*L$. Hence, we conclude that $R_1R_2 = R_2R_1$. To prove that $R_2R_3 = R_3R_2,$ we see that
		\begin{equation}\label{R_2R_3, R_3R_2}
			\begin{aligned}
				R_2R_3 &=
				\begin{pmatrix}
					S_{\alpha}P & 0 & 0 & 0 & \dots\\
					F^*D_PP + FD_P & 0 & 0 & 0 & \dots\\
					F^*D_P & F & 0 & 0 & \dots\\
					0 & F^* & F & 0 & \dots\\
					\vdots & \vdots & \vdots & \vdots & \ddots
				\end{pmatrix}
				\,\, \text{and} \hspace{0.3cm}
				R_3R_2 =
				\begin{pmatrix}
					PS_{\alpha} & 0 & 0 & 0 & \dots\\
					D_PS_{\alpha} & 0 & 0 & 0 & \dots\\
					F^*D_P & F & 0 & 0 & \dots\\
					0 & F^* & F & 0 & \dots\\
					\vdots & \vdots & \vdots & \vdots & \ddots
				\end{pmatrix}.
			\end{aligned}
		\end{equation}
In order to show $R_2R_3 = R_3R_2$, we nee to verify  $F^*D_PP + FD_P = D_PS_{\alpha}$. We observe that
		\begin{equation}\label{}
			\begin{aligned}
				F^*D_PP + FD_P &=  S_{\alpha}  = D_PS_{\alpha}.
			\end{aligned}
		\end{equation}
Thus, the $(2,1)$ entries of $R_2R_3$ and $R_3R_2$ are equal and hence $R_2R_3 = R_3R_2$. Note that 
		\begin{equation}\label{R_1R_3, R_3R_1}
			\begin{aligned}
				R_1R_3 &=
				\begin{pmatrix}
					A_{\alpha}P & 0 & 0 & 0 & \dots\\
					LD_P & 0 & 0 & 0 & \dots\\
					0 & L & 0 & 0 & \dots\\
					0 & 0 & L & 0 & \dots\\
					\vdots & \vdots & \vdots & \vdots & \ddots
				\end{pmatrix}
				\,\, \text{and} \hspace{0.3cm}
				R_3R_1 =
				\begin{pmatrix}
					PA_{\alpha} & 0 & 0 & 0 & \dots\\
					D_PA_{\alpha} & 0 & 0 & 0 & \dots\\
					0 & L & 0 & 0 & \dots\\
					0 & 0 & L & 0 & \dots\\
					\vdots & \vdots & \vdots & \vdots & \ddots
				\end{pmatrix}.
			\end{aligned}
		\end{equation}
It implies from Lemma \ref{g1112} that $LD_P = D_PA_{\alpha}$. This shows that $R_1R_3=R_3R_1$.
		
\noindent{\textbf{Step 2:}} In order to prove that $(R_2, R_3)$ is a $\Gamma$-isometry, we need to verify  $R_2 = R^*_2R_3,$ $R_3$ is an isometry and the spectral radius $r(R_2) \leqslant 2$.  We first show that $R_2 = R^*_2R_3$. Notice that
		\begin{equation}\label{R_2 star R_3}
			\begin{aligned}
				R^*_2R_3 &=
				\begin{pmatrix}
					S^*_{\alpha} & D_PF & 0 & 0 & \dots\\
					0 & F^* & F & 0 & \dots\\
					0 & 0 & F^* & F & \dots\\
					0 & 0 & 0 & F^* & \dots\\
					\vdots & \vdots & \vdots &\vdots & \ddots
				\end{pmatrix}
				\begin{pmatrix}
					P & 0 & 0 & 0 & \dots\\
					D_P & 0 & 0 & 0 & \dots\\
					0 & I & 0 & 0 & \dots\\
					0 & 0 & I & 0 & \dots\\
					\vdots & \vdots & \vdots & \vdots & \ddots
				\end{pmatrix}\\
				&=
				\begin{pmatrix}
					S^*_{\alpha}P + D_PFD_P & 0 & 0 & 0 & \dots\\
					F^*D_P & F & 0 & 0 & \dots\\
					0 & F^* & F & 0 & \dots\\
					0 & 0 & F^* & F & \dots\\
					\vdots & \vdots & \vdots &\vdots & \ddots
				\end{pmatrix}\\
				&=
				\begin{pmatrix}
					S_{\alpha} & 0 & 0 & 0 & \dots\\
					F^*D_P & F & 0 & 0 &\dots\\
					0 & F^* & F & 0 & \dots\\
					0 & 0 & F^* & F & \dots\\
					\vdots & \vdots & \vdots & \vdots & \ddots
				\end{pmatrix} \,\, (\text{by $(\ref{S alpha, P Fundamental})$})\\
				&= R_2.
			\end{aligned}
		\end{equation}
As $F^2=0$, we see that
		\begin{equation}\label{R_2 star R_2}
			\begin{aligned}
				R^*_2R_2 &=
				\begin{pmatrix}
					S^*_{\alpha} & D_PF & 0 & 0 & \dots\\
					0 & F^* & F & 0 & \dots\\
					0 & 0 & F^* & F & \dots\\
					0 & 0 & 0 & F^* & \dots\\
					\vdots & \vdots & \vdots &\vdots & \ddots
				\end{pmatrix}
				\begin{pmatrix}
					S_{\alpha} & 0 & 0 & 0 & \dots\\
					F^*D_P & F & 0 & 0 &\dots\\
					0 & F^* & F & 0 & \dots\\
					0 & 0 & F^* & F & \dots\\
					\vdots & \vdots & \vdots & \vdots & \ddots
				\end{pmatrix}\\
				&=
				\begin{pmatrix}
					S^*_{\alpha}S_{\alpha} + D_PFF^*D_P & D_PF^2 & 0 & 0 & \dots\\
					F^{*2}D_P & F^*F + FF^* & F^2 & 0 &\dots\\
					0 & F^{*2} & F^*F + FF^* & F^2 & \dots\\
					0 & 0 & F^{*2} & F^*F + FF^* & \dots\\
					\vdots & \vdots & \vdots & \vdots & \ddots
				\end{pmatrix}\\
				&=
				\begin{pmatrix}
					F^*F + FF^* & 0 & 0 & 0 & \dots\\
					0 & F^*F + FF^* & 0 & 0 &\dots\\
					0 & 0 & F^*F + FF^* & 0 & \dots\\
					0 & 0 & 0 & F^*F + FF^* & \dots\\
					\vdots & \vdots & \vdots & \vdots & \ddots
				\end{pmatrix}.
			\end{aligned}
		\end{equation}
We observe that
		\begin{equation}\label{r1112}
			\begin{aligned}
				||R_2||^2 &=
				||R^*_2R_2||\\
				&= ||F^*F + FF^*||\\
				&= ||G^*G + GG^*||\\
				&= \left|\left|
				\begin{pmatrix}
					|\alpha|^2 & 0\\
					0 & |\alpha|^2
				\end{pmatrix}\right|\right|\\
				&= |\alpha|^2\\
				&\leqslant 1.
			\end{aligned}
		\end{equation}
From \eqref{r1112}, we deduce that $||R_2|| \leqslant 2$. This shows that  $(R_2, R_3)$ is a $\Gamma$-isometry. 

\noindent{\textbf{Step 3:}}  We now  prove that $R^*_1R_1 = I - \frac{1}{4}R^*_2R_2$.  We note  that
		\begin{equation}\label{A alpha star A alpha}
			\begin{aligned}
				A^*_{\alpha}A_{\alpha} &=
				\begin{pmatrix}
					(I_{l^2(\mathbb{C}^2)} - \frac{1}{4}(G^*G + GG^*)) & 0 & 0 & 0\\
					0 & I_{l^2(\mathbb{C}^2)} & 0 & 0\\
					0 & 0 & I_{l^2(\mathbb{C}^2)} & 0\\
					0 & 0 & 0 & I_{l^2(\mathbb{C}^2)}
				\end{pmatrix}\\
				&=
				I_{\mathcal{H}} - \frac{1}{4}(S^*_{\alpha}S_{\alpha} + S_{\alpha}S^*_{\alpha})\\
				&=
				I_{\mathcal{H}} - \frac{1}{4}(F^*F + FF^*).
			\end{aligned}
		\end{equation}
In order to show $R^*_1R_1 = I - \frac{1}{4}R^*_2R_2$, it follows from \eqref{A alpha star A alpha} that
		\begin{equation}\label{R_1 star R_1}
			\begin{aligned}
				R^*_1R_1 &=
				\begin{pmatrix}
					A^*_{\alpha}A_{\alpha} & 0 & 0 & 0 & \dots\\
					0 & L^2 & 0 & 0 &\dots\\
					0 & 0 & L^2 & 0 & \dots\\
					0 & 0 & 0 & L^2 & \dots\\
					\vdots & \vdots & \vdots & \vdots & \ddots
				\end{pmatrix}\\
				&=
				\begin{pmatrix}
					I_{\mathcal{H}} - \frac{1}{4}(F^*F + FF^*) & 0 & 0 & 0 & \dots\\
					0 & L^2 & 0 & 0 & \dots\\
					0 & 0 & L^2 & 0 & \dots\\
					0 & 0 & 0 & L^2 & \dots\\
					\vdots & \vdots & \vdots & \vdots & \ddots
				\end{pmatrix}\\
				&= I - \frac{1}{4}R^*_2R_2.
			\end{aligned}
		\end{equation}
This shows that  $(R_1, R_2, R_3)$ is a $\mathcal{\bar{P}}$-isometry. Because $(A_{\alpha}, S_{\alpha}, P)=(R_1, R_2, R_3)_{|_{\mathcal{H}}},$ we conclude that $(A_{\alpha}, S_{\alpha}, P)$ is a of $\mathcal{\bar{P}}$-contraction for all $\alpha \in \overline{\mathbb D}.$
	\end{ex}
	
\textsl{Acknowledgements:}
The first-named author is supported by the research project of SERB with ANRF File Number: CRG/2022/003058, by the Science and Engineering Research Board (SERB), Department of Science and Technology (DST), Government of India.

	%\vspace{.5cm}
	%\noindent (Keshari) \sc{An OCC of Homi Bhabha National Institute, School of Mathematical Sciences, National Institute
		%	of Science Education and Research, Bhubaneswar, Post-Jatni, Khurda, 752050, India}\\
	%{E-mail address:} {\bf dinesh@niser.ac.in}
\end{document}